\documentclass[sn-mathphys-num]{sn-jnl}
\usepackage{enumitem}
\usepackage{tabularx}
\usepackage{booktabs}
\usepackage{graphicx}%
\usepackage{multirow}%
\usepackage{amsmath,amssymb,amsfonts}%
\usepackage{amsthm}%
\usepackage{mathrsfs}%
\usepackage[title]{appendix}%
\usepackage{xcolor}%
\usepackage{textcomp}%
\usepackage{manyfoot}%
\usepackage{booktabs}%
\usepackage{algorithm}%
\usepackage{algorithmicx}%
\usepackage{algpseudocode}%
\usepackage{listings}%
\usepackage{listings}%
\usepackage{supertabular}%
\usepackage{setspace}
\usepackage{caption}%
\usepackage{cleveref}%
\usepackage{comment}%
\usepackage{float}%
\usepackage{indentfirst}%
\usepackage{extarrows}%
\usepackage{subfigure}%
\usepackage{longtable}%
\usepackage{bm}%
\usepackage{pgfplots}%
\pgfplotsset{compat=1.18}%
\usepackage{hyperref}%
\usepackage{array}%
\usepackage{multirow}%
\newtheorem{corollary}{Corollary}%

\theoremstyle{thmstyleone}%
\newtheorem{theorem}{Theorem}

\theoremstyle{thmstyletwo}%
\newtheorem{example}{Example}%
\newtheorem{remark}{Remark}%

\theoremstyle{thmstylethree}%

\begin{document}

\title[Article Title]{Dynamic-Threshold Algorithms for the Continuous Quadratic Knapsack Problem: Reset Mechanisms and Complexity}

\author[1,2]{\fnm{Yong-Jin} \sur{Liu}}\email{yjliu@fzu.edu.cn}

\author[2]{\fnm{Peicheng} \sur{Xie}}\email{240320012@fzu.edu.cn}

\author*[2]{\fnm{Chuan} \sur{Yang}}\email{chuanyang@fzu.edu.cn}

\affil[1]{\orgdiv{Center for Applied Mathematics of Fujian Province, School of Mathematics and Statistics}, \orgname{Fuzhou University}, \orgaddress{\street{No. 2 Wulongjiang North Avenue}, \city{Fuzhou}, \postcode{350108}, \state{Fujian}, \country{China}}}
\affil[2]{\orgdiv{School of Mathematics and Statistics}, \orgname{Fuzhou University}, \orgaddress{\street{No. 2 Wulongjiang North Avenue}, \city{Fuzhou}, \postcode{350108}, \state{Fujian}, \country{China}}}


\abstract{
Condat's algorithm is an efficient dynamic-threshold method for projection onto the simplex, but its extension to weighted equality constraints and the algorithmic roles of resetting and removal have received limited analysis. We develop a dynamic-threshold algorithm (DTA) for a continuous quadratic knapsack problem with a weighted equality constraint. 
DTA maintains a threshold invariant through three operations--addition, reset, and removal--and we establish its finite termination and correctness. A sufficient condition under which reset cannot occur motivates a simpler no-reset variant, NDTA. We construct instances for which DTA runs in \(\Theta(n)\) time whereas NDTA requires $\Theta(n^2)$ time, although both algorithms have quadratic worst-case complexity. We further show that, when the weight ratio and the number of deletions per removal pass are bounded, a linear number of passes with positive threshold increments requires the minimum nonzero gap between input values, normalized by the data range, to be at most \(\exp[-\Theta(n\log n)]\). Numerical experiments with up to \(10^7\) variables demonstrate that DTA and NDTA achieve approximately linear empirical scaling, and outperform Secant, WMVA, Variable Fixing, Newton, Median Search, Heap, and Sort in running time.
}



\maketitle

\section{Introduction}\label{sec.1}

In this paper, we consider the following singly constrained quadratic program:
\begin{align*}\label{P}
    \min_{x \in \mathbb{R}^n} \quad & \sum_{i=1}^{n} w_i x_i^2 \\
    \text{s.t.} \quad & \sum_{i=1}^{n} w_i x_i = c, \tag{P}\\
    & x \leq b,
\end{align*}
where \(w \in \mathbb{R}^n_{++}\), \(b \in \mathbb{R}^n\), and \(c \in \mathbb{R}\) satisfy \(c < \sum_{i=1}^n w_i b_i\). Problem~\eqref{P} is a continuous quadratic knapsack problem, equivalently a singly constrained quadratic program (see~\cite{5,7,8,20}).

\subsection{Applications}

Problem~\eqref{P} arises in several settings. We mention two representative examples.

\textbf{Zero-lower-bound weighted minimum variance allocation (ZWMVA) problem.}
The weighted minimum variance allocation (WMVA) problem was introduced by Sun and Sheu~\cite{17}. Its zero-lower-bound variant arises in applications where negative allocations are excluded, such as portfolio optimization~\cite{Pang80,Bretthauer02b,15,19}, single-link projection problem~\cite{28}, and single-quadratic data fitting~\cite{Kamesam09}. The model is formulated as follows:
\begin{align*}\label{ZWMVA}
    \min_{x \in \mathbb{R}^n} \quad & \sum_{i=1}^n w_i \Bigl( x_i - \frac{c}{W_{sum}} \Bigr)^2 \\
    \text{s.t.} \quad & \sum_{i=1}^{n} w_i x_i = c, \tag{ZWMVA}\\
    & 0 \leq x \leq b,
\end{align*}
where \(w,b \in \mathbb{R}^n_{++}\), \(c>0\), and \(W_{sum}=\sum_{i=1}^n w_i\). Since \(\sum_{i=1}^n w_i \left(x_i-\frac{c}{W_{sum}}\right)^2
=
\sum_{i=1}^n w_i x_i^2-\frac{c^2}{W_{sum}}\), problem~\eqref{ZWMVA} differs from~\eqref{P} only by the constraint \(x\geq 0\). Theorem~\ref{thm:P-ZWMVA} shows that these two problems have the same set of optimal solutions.

\textbf{Continuous-relaxation of the sensor placement (CSP) problem.}
The sensor placement problem was introduced by Agnetis et al.~\cite{Agnetis12} and was further studied by Frangioni et al.~\cite{Frangioni11}. It takes the following form:
\begin{align*}\label{SP}
    \min \quad & \sum_{i=1}^n p_i z_i^2 + \sum_{i=1}^n q_i y_i, \\
    \text{s.t.} \quad & \sum_{i=1}^n z_i = 1, \tag{SP} \\
    & 0 \leq z_i \leq y_i,\ y_i \in \{0,1\},\ i=1,\ldots,n,
\end{align*}
where \(p \in \mathbb{R}^n_{++}\), \(q \in \mathbb{R}^n\), and for $i=1,\ldots,n$, \(z_i\) denotes the fraction of demand assigned to the sensor \(i\), and \(y_i\) indicates whether the sensor \(i\) is active. In continuous relaxation, the integrality constraint on \(y\) is removed. As noted in~\cite[Section~6.3]{19}, we may assume \(q_i>0\) for all $i=1,\ldots,n$; if not, then we can fix \(y_i=1\) and eliminate \(y_i\) from the problem. Therefore, the continuous-relaxation of the problem~\eqref{SP} has the following form:
\begin{align*}\label{CSP}
    \min \quad & \sum_{i=1}^n p_i z_i^2 + \sum_{i=1}^n q_i z_i, \\
    \text{s.t.} \quad & \sum_{i=1}^n z_i = 1, \tag{CSP}\\
    & z_i \geq 0.
\end{align*}
Theorem~\ref{theorem*:equivalence} shows that problem~\eqref{CSP} is equivalent to problem~\eqref{P}.

\subsection{Related Work}

Problem~\eqref{P} belongs to the classical family of continuous quadratic knapsack problems. Helgason et al.~\cite{5} gave an \(\Theta(n\log n)\) algorithm based on sorting the breakpoints. Brucker~\cite{6} and, independently, Calamai and Mor\'{e}~\cite{7} reduced the complexity to \(\Theta(n)\) by a binary search on the dual multiplier. Pardalos and Kovoor~\cite{8} further improved this approach by using median selection to fix a larger fraction of variables in each iteration. These works established the threshold structure of the solution and the associated strategy of shrinking an active set induced by the breakpoints of the piecewise-linear constraint function.

Several later methods can be regarded as refinements of this threshold search perspective. Dai and Fletcher~\cite{9} proposed a secant-type scheme with bracketing and refinement, avoiding derivative information. Kiwiel~\cite{11,12} proposed both an \(\Theta(n)\) breakpoint search method and a variable-fixing scheme based on a relaxed subproblem, and the latter has worst-case complexity \(\Theta(n^2)\). Cominetti et al.~\cite{10} reformulated the threshold equation as a one-dimensional semismooth equation and derived a Newton-type method with very fast local convergence.

A related direction arises from the WMVA problem introduced by Sun and Sheu~\cite{17}. They showed that the WMVA problem is essentially equivalent to a continuous quadratic knapsack problem and characterized its solution through the uniform distribution property. Their algorithm is effective, but still requires a full scan of the unfixed variables at each iteration, resulting in worst-case complexity \(\Theta(n^2)\); moreover, the argument excluding this worst case relies on an integrality assumption on the data.

Another relevant line of work is projection onto the simplex. Michelot~\cite{21} developed a finite algorithm for projecting a point onto the canonical simplex. The method successively localizes the solution, and at each iteration, it reduces the problem to an affine subspace through explicit and simple computations. Building on this viewpoint, Condat~\cite{18} proposed a Gauss--Seidel-like refinement in which the threshold is updated as elements are read, rather than only after a full pass. The algorithm maintains an active set of elements governed by the current threshold, while the remaining elements are fixed at their bounds. The threshold is nondecreasing throughout the procedure. When a newly read element yields a substantially better lower bound on the threshold, a reset step replaces the current active set by this element and defers the previous active set for later reconsideration. The deferred elements are then scanned once and may be added back to the active set. The algorithm finally enters a removal phase, where elements violating the current threshold are deleted iteratively, with the threshold updated after each deletion. This removal phase can yield the \(\Theta(n^2)\) worst-case complexity on certain instances, although the entire algorithm is typically much faster in practice. These facts make Condat's framework a natural starting point for our study and lead to two questions:
\begin{itemize}
    \item What structural conditions are needed for the removal step to reach the \(\Theta(n^2)\) worst case?
    \item How does the reset step affect the complexity and robustness of the algorithm?
\end{itemize}

\subsection{Our Contributions}

The main contributions of this paper are summarized as follows.

\begin{itemize}
    \item[(i)] Motivated by the dynamic-threshold framework underlying Condat's algorithm for projection onto the simplex, we develop a Dynamic-Threshold Algorithm (DTA) for problem~\eqref{P}. The proposed method applies to a class of continuous quadratic knapsack problems with a weighted equality constraint, including the ZWMVA problem and the CSP problem as special cases. We establish the correctness of DTA and study its complexity.

    \item[(ii)] We identify three basic operations governing the dynamics of DTA---addition, reset, and removal---and analyze them in detail. For the removal operation, we derive a necessary condition for the most unfavorable pattern, namely, the regime in which the active set decreases by one element per round over many consecutive rounds. For the reset operation, we first give a sufficient condition under which no reset occurs and introduce the corresponding no-reset variant, NDTA. We then show that the reset step can produce a strict complexity separation: DTA runs in \(\Theta(n)\) time on some instances for which NDTA requires \(\Theta(n^2)\), although both variants require \(\Theta(n^2)\) time on some other inputs. We further show that, when the weight ratio and the number of deletions per iteration are bounded, sustaining a linear number of removal iterations with strictly positive threshold increments forces the minimum nonzero gap among the input values to be super-exponentially small relative to their overall range. Such scale separation makes the quadratic pattern difficult to observe numerically.

    \item[(iii)] We perform extensive numerical experiments on large-scale instances with dimensions up to \(10^7\) for the two application problems~\eqref{ZWMVA} and~\eqref{CSP}. The results show a near-linear empirical scaling for both DTA and NDTA. Their running times are very close, with NDTA often being slightly faster. Both algorithms also outperform the tested algorithms including Secant~\cite{9}, WMVA~\cite{17}, Variable Fixing~\cite{12}, Newton~\cite{10}, Median Search~\cite{11}, Heap~\cite{33}, and Sort~\cite{32} in terms of running time.
\end{itemize}

\subsection{Organization of the Paper}
The remainder of the paper is organized as follows.
Section~\ref{sec.2} presents basic properties of problem~\eqref{P} and relates it to weighted projection and allocation models.
Section~\ref{sec.3} introduces DTA, describes its main operations, and proves finite termination and correctness.
Section~\ref{sec.4} studies the reset and removal mechanisms. It introduces the no-reset variant NDTA, compares the time complexity of DTA and NDTA, gives worst-case constructions, and presents instances that separate the two variants.
Section~\ref{sec.5} reports numerical results for \eqref{ZWMVA} and \eqref{CSP}.
Section~\ref{sec.6} concludes the paper and discusses possible directions for future work.

	\section{Basic Properties of Problem~\eqref{P}}\label{sec.2}

The optimality conditions for problem~\eqref{P} are summarized below.

\begin{theorem}\label{theorem 2.1}
Assume that problem~\eqref{P} is feasible. For any \(\tau\in\mathbb{R}\), define
\begin{equation}\label{eq:xgtau}
x(\tau)=\min\{\tau,b\},
\quad
g(\tau)=\sum_{i=1}^n w_i x_i(\tau),
\end{equation}
where the minimum is taken componentwise. Then \(x^*\in\mathbb{R}^n\) is an optimal solution to problem~\eqref{P} if and only if there exists \(\tau^*\in\mathbb{R}\) such that
\[
g(\tau^*)=c
\quad \text{and} \quad
x^*=x(\tau^*).
\]
Moreover, \(g\) is a nondecreasing piecewise linear function, and \(\tau^*\) is unique.
\end{theorem}

Theorem~\ref{theorem 2.1} reduces problem~\eqref{P} to the identification of a scalar threshold \(\tau\in\mathbb{R}\) such that \(g(\tau)=c\). Hence, most iterative methods proceed by evaluating \(g(\tau)\) and checking it against \(c\). On the other hand, problem~\eqref{P} also admits a weighted projection reformulation, which connects it to the Condat-type framework.

\medskip

\begin{theorem}\label{theorem*:equivalence}
Let \(\mathcal N=\{1,\ldots,n\}\). Under the change of variables
\begin{equation}\label{eq:zi_xi1}
\overline z_i=w_i(b_i-x_i),
\quad i\in\mathcal N,
\end{equation}
problem~\eqref{P} is equivalent to
\begin{align*}
\min_{\overline z\in\mathbb R^n}\quad
&\frac12\sum_{i=1}^n
\overline w_i(\overline z_i-\overline y_i)^2\\
\text{s.t.}\quad
&\sum_{i=1}^n\overline z_i=a,
\\
&\overline z\ge0,
\end{align*}
where
\begin{equation}\label{eq:variables0}
\overline w_i=\frac1{w_i},
\quad
\overline y_i=w_i b_i,
\quad
a=\sum_{i=1}^n w_i b_i-c>0.
\end{equation}
The transformation is a bijection between the feasible sets and
preserves optimality.

Moreover, suppose \(p_i>0\) for all \(i\in\mathcal N\). Then
problem~\eqref{CSP} is equivalent to problem~\eqref{P} with
\begin{equation}\label{eq:csp-to-p}
w_i=\frac1{2p_i},
\quad
b_i=-q_i,
\quad
c=-1-\sum_{i=1}^n\frac{q_i}{2p_i},
\end{equation}
under the change of variables
\begin{equation}\label{eq:csp-change}
x_i=-q_i-2p_i z_i,
\quad i\in\mathcal N.
\end{equation}
\end{theorem}

\begin{proof}
Under \eqref{eq:zi_xi1} and \eqref{eq:variables0},
\[
w_i x_i^2
=
w_i\left(b_i-\frac{\overline z_i}{w_i}\right)^2
=
\overline w_i(\overline z_i-\overline y_i)^2.
\]
Thus, the two objective functions differ only by the positive factor
\(1/2\). Moreover,
\[
\sum_{i=1}^n w_i x_i=c
\iff
\sum_{i=1}^n\overline z_i
=
\sum_{i=1}^n w_i b_i-c
=a,
\]
and
\[
x_i\le b_i
\Longleftrightarrow
\overline{z}_i\ge 0,\quad i\in\mathcal{N}.
\]
The inverse transformation is
\[
x_i=b_i-\frac{\overline z_i}{w_i},
\]
so the correspondence is bijective and preserves optimality.

For problem~\eqref{CSP}, choose the parameters in
\eqref{eq:csp-to-p}. Then \(\sum_{i=1}^n w_i b_i-c=1>0\),
so the assumptions of problem~\eqref{P} are satisfied. Under
\eqref{eq:csp-change},
\[
x_i\le b_i
\iff
-q_i-2p_i z_i\le-q_i
\iff
z_i\ge0,
\]
and
\[
\sum_{i=1}^n w_i x_i
=
-\sum_{i=1}^n\frac{q_i}{2p_i}
-\sum_{i=1}^n z_i
=c
\iff
\sum_{i=1}^n z_i=1.
\]
Finally,
\[
\sum_{i=1}^n w_i x_i^2=
\sum_{i=1}^n
\frac{(q_i+2p_i z_i)^2}{2p_i}=
2\sum_{i=1}^n(p_i z_i^2+q_i z_i)
+\sum_{i=1}^n\frac{q_i^2}{2p_i}.
\]
The last term is constant, and the remaining factor is positive.
Therefore, the transformation preserves minimizers, proving the
equivalence of problems~\eqref{CSP} and~\eqref{P}.
\end{proof}

Next, we clarify the relation between problem~\eqref{P} and the ZWMVA problem. Under the assumptions \(b\in\mathbb{R}^n_{++}\) and \(c>0\), the nonnegativity constraints are automatically satisfied at optimality. Hence, these two problems share the same set of optimal solutions.
\medskip
\begin{theorem}\label{thm:P-ZWMVA}

Assume that \(b\in\mathbb{R}^n_{++}\) and \(c>0\). Then every optimal solution of problem~\eqref{P} is nonnegative. Consequently, problems~\eqref{P} and~\eqref{ZWMVA} have the same set of optimal solutions.
\end{theorem}

\begin{proof}

Suppose for the purpose of contradiction that \(x^*\) is an optimal solution of problem~\eqref{P} and that \(x_j^*<0\) for some \(j\in\mathcal{N}\). Since \(x^*\) is feasible, we have \(\sum_{i=1}^n w_i x_i^*=c\) and hence
\[
\sigma=\sum_{i\ne j} w_i x_i^*=c-w_jx_j^*>c.
\]
Define \(\bar x\in\mathbb{R}^n\) by
\[
\bar x_i=
\begin{cases}
\dfrac{c}{\sigma}x_i^*, & i\ne j,\\[1ex]
0, & i=j,
\end{cases}\quad i =1,\ldots,n.
\]
Then we obtain
\[
\sum_{i=1}^n w_i \bar x_i
=
\sum_{i\ne j} w_i \frac{c}{\sigma}x_i^*
=
\frac{c}{\sigma}\sum_{i\ne j} w_i x_i^*
=c,
\]
which implies that \(\bar{x}\) satisfies the equality constraint. Moreover, since \(0<c/\sigma<1\), for every \(i\ne j\), if \(x_i^*\ge0\), then
\(\bar x_i=(c/\sigma)x_i^*\le x_i^*\le b_i\); if \(x_i^*<0\), then
\(\bar x_i<0<b_i\). Moreover, \(\bar x_j=0<b_j\). Hence
\(\bar x\le b\), and therefore \(\bar x\) is feasible for problem~\eqref{P}.

Note that
\[
\sum_{i=1}^n w_i \bar x_i^2
=
\sum_{i\ne j} w_i\left(\frac{c}{\sigma}x_i^*\right)^2
=
\left(\frac{c}{\sigma}\right)^2 \sum_{i\ne j} w_i (x_i^*)^2
<
\sum_{i=1}^n w_i (x_i^*)^2,
\]
which contradicts the optimality of \(x^*\). Therefore, every optimal solution of problem~\eqref{P} is nonnegative.
\end{proof}

\section{Dynamic-Threshold Algorithm}\label{sec.3}

    We now turn to DTA for problem~\eqref{P} and prove its correctness. Drawing on the nondecreasing threshold updates in Condat's efficient algorithm, the proposed DTA maintains an active set \(\mathcal S\subseteq\mathcal N\) and a threshold \(\tau\) satisfying the invariant
\begin{equation}\label{eq:1}
\sum_{i\in\mathcal S} w_i(b_i-\tau)=D,
\quad
D=\sum_{i=1}^n w_i b_i-c>0.
\end{equation}

DTA starts from an initial singleton active set, for instance \(\mathcal{S} = \{1\}\) with \(\tau = (w_1 b_1 - D)/w_1\). In subsequent iterations, it dynamically compares the current threshold \(\tau\) with the upper bound \(b_i\) to update \(\mathcal{S}\) and \(\tau\).
At termination, the active set \(\mathcal S\) is expected to coincide with
\(
\{\,i\in\mathcal N:\ b_i>\tau\,\},
\)
in which case the associated optimal solution is recovered as
\(x_i^*=\tau\ (i\in\mathcal{S})\), \(x_i^*=b_i\ (i\notin\mathcal{S})\).

\begin{algorithm}
	\caption{\large \bfseries  Dynamic-Threshold Algorithm (DTA)}
	\vspace*{2pt} 
	\label{algorithm}
	\textbf{Input:} $w \in \mathbb{R}^n_{++}$, $b \in \mathbb{R}^n$, $c\in \mathbb{R}$, and $ c < \sum_{i=1}^n w_i b_i$.
	
	\textbf{Output:} optimal solution $x^*$.
	\begin{enumerate}[nosep,leftmargin=*]

		\item Set $\mathcal{J} = \emptyset$, $\mathcal{S} = \{1\}$, 
		$D = \sum_{i=1}^{n} w_i b_i - c$,  $W= w_1$, 
		$\tau = \dfrac{w_1 b_1 - D}{W}$.
		\item For $i = 2,\ldots,n$, do:
        \begin{itemize}
			\item[] If $b_i > \tau$, set $W=W + w_i$, $\displaystyle \tau = \tau + \frac{w_i(b_i-\tau)}{W }$.
			\begin{itemize}
				\item[] If $ \tau>  \rho_i = b_i - D/w_i$, set $\mathcal{S}= \mathcal{S} \cup \{i\}$;
				  else, 
				 
				 set $\mathcal{J} =\mathcal{J}\cup \mathcal{S}$, $\mathcal{S} = \{i\}$, $W= w_i$, $\tau = \rho_i$.
			\end{itemize}
		\end{itemize}
		\item If $\mathcal{J} \neq \emptyset$, for every element $j \in \mathcal{J}$:
		\begin{itemize}
			\item[] If $b_j > \tau$, set $\mathcal{S}= \mathcal{S} \cup \{j\}$ and $W=W + w_j$,
			$\displaystyle \tau = \tau + \frac{w_j(b_j-\tau)}{W }$.
		\end{itemize}
		\item Do, while $|\mathcal{S}| $ changes, for every element $i \in \mathcal{S}$:
		\begin{itemize}
			\item[] If $b_i \le \tau$, set $\mathcal{S}= \mathcal{S} \setminus  \{i\}$ and  $W=W - w_i$, 
			$\displaystyle \tau = \tau + \frac{w_i(\tau - b_i)}{W  }$.
		\end{itemize}
		\item For $i = 1,2,\ldots,n$, return \(x_i^*=\tau\ (i\in\mathcal{S})\), \(x_i^*=b_i\ (i\notin\mathcal{S})\).
	\end{enumerate}
\end{algorithm}

According to Algorithm~\ref{algorithm}, the evolution of DTA is governed by three basic operations: addition, reset, and removal.

\smallskip
\noindent
\textbf{Addition (Steps~2 and 3).}
An addition occurs whenever an index is inserted into the active set \(\mathcal S\). 
In Step~2, if the current index \(i\) satisfies \(b_i>\tau\) but does not trigger a reset, then \(i\) is added to \(\mathcal S\), and both \(W\) and the threshold $\tau$ are updated.
In Step~3, the same update is applied to an index \(j\in\mathcal J\) from the buffer set whenever \(b_j>\tau\). It is easily observed that every addition strictly increases \(\tau\).

\smallskip
\noindent
\textbf{Reset (Step~2).}
For each newly examined index \(i\), if \(b_i>\tau\) and \(\rho_i\ge \tau\), then the current active set \(\mathcal S\) is moved into the buffer set \(\mathcal J\), the active set is reset to the singleton \(\{i\}\), \(W\) is set to \(w_i\), and the threshold is updated to \(\rho_i\). Hence, every reset strictly increases \(\tau\).

\smallskip
\noindent
\textbf{Removal (Step~4).}
The removal occurs when an index \(i\in\mathcal S\) satisfies \(b_i\le\tau\). In this case, \(i\) is removed from \(\mathcal S\), \(W\) and the threshold $\tau$ are also updated. Since \(\tau-b_i\ge 0\) is satisfied at removal, this operation cannot decrease \(\tau\); it is strictly increasing whenever \(b_i<\tau\).

Consequently, every update of DTA yields a nondecreasing threshold, while every addition and every reset increase \(\tau\) strictly. This monotonicity property is the key ingredient in the convergence analysis.

\begin{theorem}
\label{thm:correct-dta}
Let \(w\in\mathbb{R}^n_{++}\), \(b\in\mathbb{R}^n\), and
\(c\in\mathbb{R}\) satisfy \(c<\sum_{i=1}^n w_i b_i\). Then DTA terminates finitely and returns an optimal solution of
problem~\eqref{P}.
\end{theorem}

\begin{proof}
We first prove finite termination. Steps~2 and 3 each consist of a finite scan. In Step~4, every
nonterminal iteration removes at least one index from \(\mathcal S\).
On the other hand, \eqref{eq:1} gives \(\sum_{i\in\mathcal S}w_i(b_i-\tau)=D>0\),
so at least one index \(i\in\mathcal S\) satisfies \(b_i>\tau\). Hence
\(\mathcal S\) never becomes empty. Step~4 therefore has at most
\(n-1\) iterations that remove indices, followed by a terminating
iteration. Thus DTA terminates finitely.

Let \((\mathcal S,\tau)\) be the active set and threshold at
termination. The stopping rule in Step~4 gives
\[
i\in\mathcal S\Longrightarrow b_i>\tau.
\]
Now take \(i\notin\mathcal S\). The index \(i\) was either skipped in
Step~2, transferred to \(\mathcal J\) in Step~2 and not reinserted into
\(\mathcal S\) when \(\mathcal J\) was scanned in Step~3, or removed
in Step~4. In each case, there was a threshold
\(\widehat{\tau}\) such that \(b_i\le\widehat{\tau}\).
Since the threshold is nondecreasing, \(b_i\le\widehat{\tau}\le\tau\).
It follows that
\[
\mathcal S=\{i\in\mathcal N:b_i>\tau\}.
\]

The solution returned by DTA is consequently
\(x_i^*=\tau\ (i\in\mathcal{S})\), \(x_i^*=b_i\ (i\notin\mathcal{S})\),
and hence \(x_i^*=\min\{b_i,\tau\}\). \eqref{eq:1} is initialized in Step~1 and preserved by all the addition, reset, and removal updates. Therefore,
\[
\sum_{i=1}^n w_i x_i^*
=
\sum_{i\notin\mathcal S} w_i b_i+\sum_{i\in\mathcal S} w_i\tau
=
\sum_{i=1}^n w_i b_i-\sum_{i\in\mathcal S} w_i(b_i-\tau)
=
\sum_{i=1}^n w_i b_i-D
=c.
\]
Thus \(g(\tau)=c\), and
Theorem~\ref{theorem 2.1} shows that \(x^*\) is optimal for
problem~\eqref{P}.
\end{proof}

\begin{remark}

At each update in the removal step, \(\tau\) increases or remains unchanged, and
\(|\mathcal S|\) decreases. The active set \(\mathcal S\) cannot become empty. Indeed, if \(\mathcal S=\{i\}\), then \(w_i(b_i-\tau)=D>0\), and hence \(\tau=b_i-D/w_i<b_i\). Thus \(i\) cannot be removed from \(\mathcal{S}\), which means that
\(|\mathcal S|\ge1\) during the removal step. Consequently, \(\tau\) never decreases during DTA, and whenever an update leaves \(\tau\) unchanged, \(|\mathcal S|\) strictly decreases.
\end{remark}

\section{Analysis of the Reset Mechanism}\label{sec.4}

The reset mechanism is one of the main features distinguishing DTA from the simpler no-reset variant. It originates from Condat's algorithm \cite{18} in the context of simplex projection, where it is intended to rapidly correct the threshold estimate at an early stage and thereby avoid unnecessary iterations. In DTA, reset serves a similar purpose: it allows the method to abandon an unfavorable active set and rebuild it around a more promising element, while still retaining the possibility of reintroducing previously discarded elements.

\subsection{When Reset Becomes Inactive}

The reset mechanism can be highly beneficial on some instances, but it is not always active and does not necessarily improve performance in every regime. We therefore begin by identifying a sufficient condition under which the reset condition \(\rho_i\ge\tau\) becomes inactive.

\begin{theorem}\label{theorem*:4.1}
Let \(w\in\mathbb{R}_{++}^n\), \(b\in\mathbb{R}^n\), \(c\in\mathbb{R}\) satisfying \(c<\sum_{i=1}^n w_i b_i\), and \(D=\sum_{i=1}^n w_i b_i-c\). Define
\[
w_{\min}=\min_i w_i,
\quad
w_{\max}=\max_i w_i,
\quad
b_{\min}=\min_i b_i,
\quad
b_{\max}=\max_i b_i,
\]
\[
\kappa=\left\lfloor \frac{w_{\max}}{w_{\min}}\right\rfloor+1,
\quad
\eta=\frac{1}{w_{\max}}-\frac{1}{\kappa w_{\min}}>0.
\]
If \(|\mathcal S|\ge \kappa\) and
\begin{equation}
\label{eq:assump0}
    b_{\max}-b_{\min}< D\eta,
\end{equation}
then no \(i\notin\mathcal S\) can satisfy the reset condition \(\rho_i\ge\tau\).
Furthermore, if \(b\in\mathbb{R}_{++}^n\) and
\begin{equation}
\label{eq:assump1}
    c=\gamma\sum_{i=1}^n w_i b_i
\quad\text{for some }\gamma\in(0,1),
\end{equation}
then the same conclusion holds whenever \(|\mathcal S|\ge \kappa\) and
\begin{equation}\label{eq:n_n0}
    n>N_0,\quad N_0=
\left\lfloor
\frac{b_{\max}-b_{\min}}
{(1-\gamma)w_{\min}b_{\min}\eta}
\right\rfloor.
\end{equation}
\end{theorem}

\begin{proof}
Suppose, for contradiction, that \(\rho_i\ge\tau\) for some \(i\notin\mathcal S\). By definition, we get
\[
b_i-\frac{D}{w_i}
\ge
\frac{\sum_{j\in\mathcal S} w_j b_j-D}{\sum_{j\in\mathcal S} w_j}.
\]
Let
\[
W_S=\sum_{j\in\mathcal S} w_j,
\quad
\bar b_S=\frac{1}{W_S}\sum_{j\in\mathcal S} w_j b_j.
\]
Then
\begin{equation}\label{eq:reset-1}
b_i-\bar b_S
\ge
D\left(\frac{1}{w_i}-\frac{1}{W_S}\right).
\end{equation}
Since \(b_i\le b_{\max}\) and \(\bar b_S\ge b_{\min}\), we have
\begin{equation}\label{eq:reset-2}
b_i-\bar b_S\le b_{\max}-b_{\min}.
\end{equation}
Moreover, \(w_i\le w_{\max}\) and \(W_S\ge |\mathcal S|w_{\min}\), and thus
\[
\frac{1}{w_i}-\frac{1}{W_S}
\ge
\frac{1}{w_{\max}}-\frac{1}{|\mathcal S|w_{\min}}
\ge
\frac{1}{w_{\max}}-\frac{1}{\kappa w_{\min}}
=
\eta.
\]
Combining this with \eqref{eq:reset-1} and \eqref{eq:reset-2} gives \(b_{\max}-b_{\min}\ge D\eta\),
contradicting \eqref{eq:assump0}. This proves the first claim. For the second claim, if \(b\in\mathbb{R}_{++}^n\) and \eqref{eq:assump1} holds, then we obtain
\[
D=(1-\gamma)\sum_{i=1}^n w_i b_i
\ge (1-\gamma)n w_{\min}b_{\min}.
\]
Hence,
\[
b_{\max}-b_{\min}
<
(1-\gamma)n w_{\min}b_{\min}\eta
\Longrightarrow b_{\max}-b_{\min}< D\eta.
\]
Therefore, by the first part, it suffices to require \eqref{eq:n_n0}.
\end{proof}

Theorem~\ref{theorem*:4.1} shows that the reset condition becomes impossible once the active set is sufficiently large. Specifically, if \(|\mathcal S|\ge \kappa\) and \(b_{\max}-b_{\min}< D\eta\), then \(\rho_i\ge\tau\) cannot hold for any \(i\notin\mathcal S\). Moreover, the theorem converts this criterion into a dimension-dependent threshold: if \(n>N_0\), then \(b_{\max}-b_{\min}< D\eta\) holds automatically, so any reset can occur only before the active set reaches size \(\kappa\). This prediction is confirmed by the uniformly distributed experiments in Section~\ref{sec.5}. In this setting,
\[
\gamma = 0.6,\quad b_{\max} = 15,\quad b_{\min} = 1,\quad w_{\max} = 25,\quad w_{\min} = 10,
\]
which give \(\kappa=3\), \(\eta=1/150\), and \(N_0=525\). Therefore, for all tested dimensions with \(n>525\), Theorem~\ref{theorem*:4.1} implies that the reset is impossible once \(|\mathcal S|\ge 3\).

To validate this, we repeated the uniformly distributed experiment $10$ times for each $n$ and recorded both the number of reset operations (denoted as \(n_r\)) and the active set size $|\mathcal{S}|$ at each reset stage (see Table~\ref{table:reset}). The results show that the number of resets is always between \(0\) and \(2\), and $|\mathcal{S}|$ at reset never exceeds \(1\), which is in full agreement with Theorem~\ref{theorem*:4.1}.

\begin{table}[htbp]
	\centering
	\caption{Ten runs of the uniformly distributed experiments: Average, maximum, and minimum numbers of reset operations (\(n_r\)) in DTA, and average, maximum, and minimum active set sizes \(|\mathcal S|\) at reset stage.}
	\label{table:reset}
	\small
	\setlength{\tabcolsep}{14pt} 
	\begin{tabular}{ccccccc}
		\toprule
		$n$ & \multicolumn{1}{c}{$5\times 10^4$} & \multicolumn{1}{c}{$10^5$} & \multicolumn{1}{c}{$5\times 10^5$} & \multicolumn{1}{c}{$10^6$} & \multicolumn{1}{c}{$5\times 10^6$} & \multicolumn{1}{c}{$10^7$} \\
		\midrule
        &\multicolumn{6}{c}{Numbers of reset operations in DTA}\\
		Avg\((n_r)\)     &  2.0 & 1.0  & 0.0 & 1.0  & 0.0 & 0.0  \\
		Max\((n_r)\)      & 2.0 & 1.0 & 0.0   & 1.0  & 0.0  & 0.0  \\
		Min\((n_r)\)      & 2.0 & 1.0 & 0.0  & 1.0 & 0.0 & 0.0  \\
        \midrule
        &\multicolumn{6}{c}{Active set size when triggering the reset} \\
		Avg\((|S|)\)     &  1.0 & 1.0  &--  & 1.0  & -- & --  \\
		Max\((|S|)\)      & 1.0 & 1.0 & -- & 1.0 & -- & --  \\
		Min\((|S|)\)      & 1.0 & 1.0 & -- & 1.0 & -- & --  \\
		\bottomrule
	\end{tabular}
\end{table}

\subsection{A No-Reset Variant of DTA}

These observations suggest that, in this setting, the reset mechanism plays only a limited role. To isolate the effect of the reset mechanism, we consider a simplified variant of DTA in which reset is completely removed. More precisely, whenever a newly encountered element satisfies \(b_i>\tau\), it is directly merged into the current active set, without checking the condition \(\rho_i\ge\tau\), without resetting \(\mathcal S\), and without storing discarded elements in the auxiliary set \(\mathcal J\). The resulting algorithm is called the No-Reset Dynamic-Threshold Algorithm (NDTA). Its finite termination and correctness are established below.

\begin{algorithm}
	\caption{\large \bfseries No-Reset Dynamic-Threshold Algorithm (NDTA)}
	\vspace*{2pt} 
	\label{algorithm2}
	\textbf{Input:} 
    $w\in\mathbb{R}^n _{++}$, $b\in\mathbb{R}^n$, $c\in \mathbb{R}$, and $ c < \sum_{i=1}^n w_i b_i$.

	\textbf{Output:} optimal solution $x^*$.
	\begin{enumerate}[nosep,leftmargin=*]

		\item Set $\mathcal{S} = \{1\}$, 
		$D = \sum_{i=1}^{n} w_i b_i - c$,  $W= w_1$, 
		$\tau = \dfrac{w_1 b_1 - D}{W}$.
		\item For $i = 2,\ldots,n$, do:
		\begin{itemize}
			\item[] If $b_i > \tau$, 			
				 set $\mathcal{S}= \mathcal{S} \cup \{i\}$ and $W=W + w_i$,
				$\displaystyle \tau = \tau + \frac{w_i(b_i-\tau)}{W }$.
			\end{itemize}
		\item Do, while $|\mathcal{S}| $ changes, for every element $i \in \mathcal{S}$:
		\begin{itemize}
			\item[] If $b_i \le \tau$, set $\mathcal{S}= \mathcal{S} \setminus  \{i\}$ and  $W=W - w_i$, 
			$\displaystyle \tau = \tau + \frac{w_i(\tau - b_i)}{W  }$.
		\end{itemize}
		\item For $i = 1,2,\ldots,n$, return \(x_i^*=\tau\ (i\in\mathcal{S})\), \(x_i^*=b_i\ (i\notin\mathcal{S})\). 
	\end{enumerate}
\end{algorithm}

\begin{theorem}
Let \(w\in\mathbb{R}^n_{++}\), \(b\in\mathbb{R}^n\), and \(c\in\mathbb{R}\) satisfy \(c<\sum_{i=1}^n w_i b_i\).
Then NDTA terminates finitely and returns an optimal solution of problem~\eqref{P}.
\end{theorem}

\begin{proof}
Step~2 is a finite scan. The termination argument for Step~4 of DTA
applies directly to Step~3 of NDTA: \eqref{eq:1} prevents
\(\mathcal S\) from becoming empty, and each nonterminal iteration
strictly decreases \(|\mathcal S|\). Hence NDTA terminates finitely.

Let \((\mathcal S,\tau)\) be the active set and threshold at
termination. The stopping rule gives \(b_i>\tau\) for every
\(i\in\mathcal S\). If \(i\notin\mathcal S\), then \(i\) was either
not added in Step~2 or removed in Step~3. In either case,
\(b_i\le\widehat{\tau}\) for some intermediate threshold
\(\widehat{\tau}\). Since the threshold is nondecreasing,
\(b_i\le\widehat{\tau}\le\tau\). Therefore, \(\mathcal S=\{i\in\mathcal N:b_i>\tau\}\).
The feasibility and optimality arguments now follow exactly as in the
proof of Theorem~\ref{thm:correct-dta}.
\end{proof}

Compared with DTA, the active set in NDTA evolves only through direct additions in Step~2 and possible deletions in the final removal step, without any intermediate rebuilding process. Therefore, NDTA provides a natural baseline for assessing the role of reset. Although this modification appears minor, it may fundamentally alter the behavior of the algorithm on certain instances.

{
\subsection{Complexity Analysis and Worst-Case Instances}
\label{Complexity Separation and Worst-Case Instances}

We now show that the reset mechanism can lead to a strict separation in computational complexity. We construct a parametric family consisting of a prefix of length \(m_1\) and a suffix of length \(m_2\), where \(n=m_1+m_2\). 
For this family of instances, the time complexities of DTA and NDTA are
\[
\Theta(m_1+m_2^2)
\quad\text{and}\quad
\Theta\bigl((m_1+m_2)^2\bigr),
\]
respectively. When \(m_2^2=\Theta(m_1)\), DTA runs in \(\Theta(n)\) time, whereas NDTA requires \(\Theta(n^2)\) time. When \(m_2=\Theta(n)\), however, DTA also requires \(\Theta(n^2)\) time. Thus, the reset mechanism can reduce the running time substantially on some instances, but it does not eliminate the quadratic worst case.

\begin{example}\label{ex:unified-parametric}
Let \(m_1\ge 3\) and \(m_2\ge 5\), and set \(n=m_1+m_2\). Consider an instance with \(w_i=1\) for all \(i=1,\ldots,n\) and \(D=1\). Partition \(b\) into a prefix \(A=(u_1,\ldots,u_{m_1})\)
and a suffix \(B=(v_1,\ldots,v_{m_2})
\). Set \(c=\sum_{i=1}^{m_1}u_i+\sum_{j=1}^{m_2}v_j-1\).
To construct \(B\), set \(C_1=1\) and
\[
C_k=n(k+1)C_{k-1},
\quad k=2,\ldots,m_2-2.
\]
Choose \(0<\epsilon<\frac{1}{2nC_{m_2-2}}\),
and define
\[
v_1=3-\epsilon C_{m_2-2},
\quad
v_2=v_3=\frac{7}{2},\quad v_i=3-\epsilon C_{i-3},
\quad i=4,\ldots,m_2.
\]
Denote \(V=\sum_{j=1}^{m_2}v_j\).
The choice of \(\epsilon\) ensures that \(v_j>2\) for every \(j=1,\ldots,m_2\).

We next construct \(A\). Set \(u_2=1\). For \(k=3,\ldots,m_1\), suppose that \(u_2,\ldots,u_{k-1}\) have been defined, and let
\(Q_k=\sum_{i=2}^{k-1}u_i+V-1\), \(q_k=m_2+k-1\).
Define
\[
u_k=
\min\left\{
u_{k-1}-1,\;
q_k u_{k-1}-Q_k-1,\;
\frac{Q_k}{q_k-1}-1,\;
-1
\right\}.
\]
Finally, set
\[
\begin{aligned}
u_1=\min\Bigg\{ -2,\;&
\min_{3\le k\le m_1}
\left((k-1)u_k-\sum_{i=2}^{k-1}u_i\right),\\
&
\min_{1\le j\le m_2}
\left(
(m_1+j-1)v_j
-\sum_{i=2}^{m_1}u_i
-\sum_{\ell=1}^{j-1}v_\ell
\right)
\Bigg\}.
\end{aligned}
\]

\end{example}

\begin{theorem}\label{thm:unified-parametric}
Consider the family of instances in Example~\ref{ex:unified-parametric}, with the elements scanned in the order \(u_1,\ldots,u_{m_1},v_1,\ldots,v_{m_2}
\).
On this family, the time complexities of DTA and NDTA are
\[
\Theta(m_1+m_2^2)
\quad\text{and}\quad
\Theta\bigl((m_1+m_2)^2\bigr),
\]
respectively. In particular, if \(m_2^2=\Theta(m_1)\), then DTA runs in
\(\Theta(m_1)\) time, whereas NDTA runs in
\(\Theta(m_1^2)\) time.
\end{theorem}}

\begin{proof}
Let \(M=m_2-2\) and \(E=\epsilon\left(C_M+\sum_{r=1}^{M-1}C_r\right)\). Since \(C_r=n(r+1)C_{r-1}\), we have
\[
\sum_{\ell=1}^{r-1}C_\ell<C_r
\quad (r\ge2) \quad\text{and} \quad E<2\epsilon C_M<\frac1n<1.
\]

We first consider DTA. Initially, \(S=\{1\}\) and \(\tau=u_1-1\).
Since \(u_1<0\), reading \(u_2=1\) triggers a reset, after which \(S=\{2\}\) and \(\tau=0\). Because \(u_k\le-1\) for \(k=3,\ldots,m_1\), none of the remaining
prefix elements enters \(S\). Since \(v_1>2\), reading \(v_1\) triggers a second reset. Thus \(S\)
is replaced by \(\{m_1+1\}\), while indices \(1\) and \(2\) are
deferred in \(\mathcal{J}\). The remaining suffix elements enter \(S\) without further
resets. Indeed, for \(j\ge4\), the threshold before \(v_j\) is read is
\[
3-\frac{\epsilon}{j-1}
\left(C_M+\sum_{r=1}^{j-4}C_r\right),
\]
which is smaller than \(v_j=3-\epsilon C_{j-3}\), since \((j-1)C_{j-3}<C_M\).
The updated threshold is greater than \(2\), whereas
\(v_j-1<2\), so no reset occurs. The cases \(j=2,3\) follow directly
from \(v_2=v_3=7/2\). At the end of the forward scan, \(S=\{m_1+1,\ldots,n\}\). Its threshold is
\[
\tau=\frac{V-1}{m_2}
    =3-\frac{E}{m_2}>2.
\]
Consequently, neither \(u_1\) nor \(u_2\) is reinserted. It remains to examine the removal step. Let \(S_k\) and \(\tau_k\)
denote, respectively, the active set and the threshold at the
beginning of iteration \(k\). At the beginning of the first iteration,
\[
S_1=\{v_1,\ldots,v_{m_2}\},
\quad
\tau_1=3-\frac{E}{m_2}.
\]
Since \(v_1=3-\epsilon C_M\le \tau_1\), the element \(v_1\) is removed. The threshold is then updated within
the same iteration to
\[
\widetilde{\tau}_1
=
3-\frac{\epsilon\sum_{r=1}^{m_2-3}C_r}{m_2-1}.
\]
The elements \(v_2,\ldots,v_{m_2-1}\) remain active, while
\(v_{m_2}<\widetilde{\tau}_1\). Thus, the first iteration eliminates
exactly \(v_1\) and \(v_{m_2}\), and \(S_2=\{v_2,\ldots,v_{m_2-1}\}\). For \(2\le k\le m_2-3\), define \(p_k=m_2-k+1\).
At the beginning of iteration \(k\),
\[
S_k=\{v_2,\ldots,v_{p_k}\},\quad \tau_k
=
3-
\frac{\epsilon\sum_{r=1}^{p_k-3}C_r}{p_k-1}.
\]
For \(4\le i<p_k\),
\[
C_{i-3}
\le C_{p_k-4}
<
\frac{C_{p_k-3}}{p_k-1}
<
\frac{\sum_{r=1}^{p_k-3}C_r}{p_k-1},
\]
hence \(v_i>\tau_k\). Moreover, \(v_2=v_3=7/2>\tau_k\).
On the other hand,
\[
\frac{\sum_{r=1}^{p_k-3}C_r}{p_k-1}
<C_{p_k-3},
\]
so \(v_{p_k}<\tau_k\). Therefore, iteration \(k\) removes exactly \(v_{p_k}\), and
\(S_{k+1}=\{v_2,\ldots,v_{p_k-1}\}\). After iteration \(m_2-3\), only \(v_2\) and \(v_3\) remain. Hence \(S_{m_2-2}=\{v_2,v_3\}\) and \(\tau_{m_2-2}=3\). Since \(v_2=v_3=7/2>3\), the final iteration removes no element and
the removal step terminates. The total number of examined elements is \(m_2+\sum_{k=2}^{m_2-3}(m_2-k)+2
=\Theta(m_2^2)\). Including the forward scan and final reconstruction, the overall time complexity of DTA is  
\[
\Theta(m_1+m_2^2).
\]

We next consider NDTA. Since no reset is performed, it suffices to
show that every element enters the active set. Before \(u_k\) is read,
where \(3\le k\le m_1\), the threshold is
\[
\tau
=
\frac{u_1+\sum_{i=2}^{k-1}u_i-1}{k-1}.
\]
The definition of \(u_1\) gives \(u_1\le (k-1)u_k-\sum_{i=2}^{k-1}u_i\), and therefore \(\tau<u_k\). Similarly, before \(v_j\) is read,
\[
u_1\le
(m_1+j-1)v_j
-\sum_{i=2}^{m_1}u_i
-\sum_{\ell=1}^{j-1}v_{\ell}
\]
implies that the current threshold is smaller than \(v_j\).
Consequently, all \(n\) elements enter the active set.

The inequalities defining \(u_k\) imply
\[
u_k<
\theta_k=
\frac{Q_k+u_k}{q_k}
<u_{k-1},
\quad 3\le k\le m_1.
\]
In the first removal round, \(u_1\) and \(u_{m_1}\) are removed.
Subsequently, whenever the active prefix is
\(\{u_2,\ldots,u_k\}\), its threshold is \(\theta_k\). Thus
\(u_2,\ldots,u_{k-1}\) are retained and \(u_k\) is removed. After
removing \(u_k\), the threshold becomes \(\theta_{k-1}\); since
\(u_{k-1}\) has already been scanned, it remains until the next
round. All suffix elements are retained during these rounds because
\(v_j>2\). Hence the prefix phase requires rounds whose active-set sizes are \(\Theta(m_2+k)\), \(k=3,\ldots,m_1\). Its cost is
\[
\Theta\left(\sum_{k=3}^{m_1}(m_2+k)\right)
=
\Theta(m_1m_2+m_1^2).
\]
After the prefix has been removed, the suffix follows the same
trajectory as for DTA and costs \(\Theta(m_2^2)\). Therefore, NDTA runs in
\[
\Theta(m_1^2+m_1m_2+m_2^2)
=
\Theta\bigl((m_1+m_2)^2\bigr)
\]
time. The proof is complete.
\end{proof}

Theorem~\ref{thm:unified-parametric} shows that the reset mechanism can
change the asymptotic complexity, although it does not eliminate the
quadratic worst case. The strict separation is already generated by
the prefix construction, while the suffix is used to recover the
quadratic behavior of DTA. By isolating the prefix mechanism, we obtain
the following simpler instance that directly separates DTA from NDTA.
\medskip

\begin{corollary}\label{cor:explicit-separation}
For every \(n\ge4\), consider the instance with \(w_i=1\) for
\(i=1,\ldots,n\) and \(D=1\). Define
\[
C_1=1,
\quad
C_k=1+\sum_{\ell=2}^{k}\ell!,
\quad k=2,\ldots,n-1,
\]
and let
\begin{equation}\label{eq:explicit-instance}
b_1=-C_{n-1},
\quad
b_2=1,
\quad
b_i=-C_{i-2},
\quad i=3,\ldots,n.
\end{equation}
Set \(c=\sum_{i=1}^{n}b_i-1\), and scan the elements in the order \(b_1,b_2,\ldots,b_n\). Then DTA
runs in \(\Theta(n)\) time, whereas NDTA runs in \(\Theta(n^2)\) time.
\end{corollary}
\medskip
For DTA, processing \(b_2=1\) triggers a reset. The active set becomes
\(\mathcal S=\{2\}\), with \(\tau=0\), while the index \(1\) is deferred.
Since \(b_i<0\) for all \(i\ge3\), none of the remaining elements
enters the active set, and \(b_1\) is not reinserted. DTA therefore
enters the removal step with a singleton active set and runs in
\(\Theta(n)\) time.

For NDTA, the reset is absent and the large negative value of \(b_1\)
keeps the threshold below every subsequently scanned element. Hence
all \(n\) elements enter the active set. In the first removal iteration, \(b_1\) and \(b_n\) are deleted. Thereafter, \(b_{n-1},\ldots,b_3\) are removed one at a time, one per iteration. Since these iterations scan active
sets of decreasing but linear size, their total cost is
\(\Theta(n^2)\).

\begin{table}[htbp]
\centering
\caption{Average number of resets (\(n_r\)), average number of elements examined
in the removal step by DTA (\(n_1\)) and NDTA (\(n_2\)), and their
cumulative running times in milliseconds (\(t_1\) and \(t_2\)) over
\(10^6\) runs.}
\label{tab:3}
\small
\setlength{\tabcolsep}{20pt}
\begin{tabularx}{\textwidth}{cccccc}
\toprule
$n$ & \multicolumn{1}{c}{$4$} & \multicolumn{1}{c}{$6$} & \multicolumn{1}{c}{$8$} & \multicolumn{1}{c}{$10$} & \multicolumn{1}{c}{$12$} \\
\midrule
Avg(\(n_r\))      & 1.0 & 1.0 & 1.0 & 1.0 & 1.0  \\
Avg(\(n_1\))      & 1.0 & 1.0 & 1.0 & 1.0 & 1.0  \\
Avg(\(n_2\))      & 7.0 & 16.0 & 29.0 & 46.0 & 67.0  \\
$t_1$ (ms) & 51 & 55 & 73 & 84 & 107 \\
$t_2$ (ms) & 59 & 70 & 104 & 137 & 184\\
$t_2/t_1$  & 1.16 & 1.27 & 1.42 & 1.63 & 1.72 \\
\bottomrule
\end{tabularx}
\end{table}

Table~\ref{tab:3} reports the cumulative running times and algorithmic
statistics over \(10^6\) runs on instance~\eqref{eq:explicit-instance}
for different values of \(n\). NDTA is consistently slower than DTA,
and the ratio \(t_2/t_1\) increases with \(n\). In each run, DTA
performs one reset and examines only one element (\(n_1=1\)) in the removal step. For NDTA, the number of examined elements is \(n_2=n(n-1)/2+1\), as predicted by its \(\Theta(n^2)\) complexity. This quadratic behavior relies on a specially constructed instance.
We next consider a removal loop that runs for \(\Theta(n)\) iterations and
removes only a bounded number of elements per iteration. The following
theorem shows that, under these assumptions, the range of the \(b_i\)
values must be large relative to their minimum nonzero separation.

\medskip

\begin{theorem}\label{thm:worst-case-necessary}
Let \(S_t\) be the active set at the beginning of iteration \(t\) of
the removal loop, and let \(m=|S_1|\). Suppose that \(D>0\) and that
the \(b_i\)'s for \(i\in S_1\) are not all equal. Suppose that, for
some positive integer \(K\) and some \(3\le s\le \left\lfloor m/K \right\rfloor\),
each of the first \(s\) iterations removes at least one and at most
\(K\) elements, and that the threshold increases strictly after each
of these iterations. Let
\[
\xi=
\frac{\max_{i\in S_1}w_i}{\min_{i\in S_1}w_i},
\quad
\delta=
\min_{\substack{i,j\in S_1\\ b_i\ne b_j}}
|b_i-b_j|,
\quad
L_b=
\max_{i\in S_1}b_i-\min_{i\in S_1}b_i.
\]
Then, it holds that
\[
\delta
<
L_b
\left(
1+\frac{K\xi}{m-(s-1)K}
\right)
\frac{(K\xi)^{s-3}}
{\displaystyle\prod_{j=2}^{s-2}(m-jK)}.
\]
For \(s=3\), the empty product is understood to be \(1\).
\end{theorem}

\begin{proof}
Define
\[
\tau_t=
\frac{\sum_{i\in S_t}w_i b_i-D}
     {\sum_{i\in S_t}w_i},
\qquad t=1,\ldots,s+1,
\]
and let \(\Delta_t=\tau_{t+1}-\tau_t\), \(t=1,\ldots,s\). By assumption, \(\Delta_t>0\) for each \(t\).
Let
\[
R_t=S_t\setminus S_{t+1},\quad
W_t=\sum_{i\in S_t}w_i,\quad
G_t=\sum_{i\in R_t}w_i,
\]
and define
\[
\bar b_t=\frac{\sum_{i\in R_t}w_i b_i}{G_t}.
\]
Since \(S_t=S_{t+1}\cup R_t\), the threshold identities for \(S_t\)
and \(S_{t+1}\) give \(W_{t+1}(\tau_{t+1}-\tau_t)
=
G_t(\tau_t-\bar b_t)\), thus
\begin{equation}\label{eq:wc-delta-identity}
\Delta_t
=
\frac{G_t}{W_{t+1}}(\tau_t-\bar b_t).
\end{equation}

For \(t\ge2\), every element of \(R_t\) survived iteration \(t-1\).
Since the threshold is nondecreasing during an iteration, \(b_i>\tau_{t-1}\), \(i\in R_t\). Moreover, \(\Delta_t>0\) and \eqref{eq:wc-delta-identity} imply
\(\bar b_t<\tau_t\). Hence
\[
\tau_{t-1}<\bar b_t<\tau_t,\quad 0<\tau_t-\bar b_t<\Delta_{t-1}.
\]
Using \eqref{eq:wc-delta-identity}, we obtain
\begin{equation}\label{eq:wc-delta-contraction}
\Delta_t<
\frac{G_t}{W_{t+1}}\Delta_{t-1}.
\end{equation}
At most \(K\) elements are removed in iteration \(t\), so \(G_t\le K\max_{i\in S_1}w_i\). After the first \(t\) iterations, at least \(m-tK\) elements remain.
Therefore, \(W_{t+1}\ge
(m-tK)\min_{i\in S_1}w_i\). It follows from \eqref{eq:wc-delta-contraction} that
\begin{equation}\label{eq:wc-delta-recursion}
\Delta_t<
\frac{K\xi}{m-tK}\Delta_{t-1},
\quad t\ge2.
\end{equation}
Iterating \eqref{eq:wc-delta-recursion} gives
\begin{equation}\label{eq:wc-delta-iteration}
\Delta_{s-1}
<
\frac{K\xi}{m-(s-1)K}\Delta_{s-2},\quad
\Delta_{s-2}
<
\Delta_1
\frac{(K\xi)^{s-3}}
{\displaystyle\prod_{j=2}^{s-2}(m-jK)}.
\end{equation}
For each \(t\ge2\), all elements of \(R_t\) exceed
\(\tau_{t-1}\), while their weighted average is smaller than
\(\tau_t\). Hence, there exists \(i_t\in R_t\) such that \(\tau_{t-1}<b_{i_t}<\tau_t\).
In particular, \(b_{i_{s-1}}<\tau_{s-1}<b_{i_s}\), so these two values are distinct, and
\[
\delta
\le b_{i_s}-b_{i_{s-1}}
<
\tau_s-\tau_{s-2}
=
\Delta_{s-2}+\Delta_{s-1}.
\]
Using \eqref{eq:wc-delta-iteration}, we obtain
\begin{equation}\label{eq:wc-delta-gap}
\delta
<
\Delta_{s-2}
\left(
1+\frac{K\xi}{m-(s-1)K}
\right).
\end{equation}
Finally, \(\Delta_1>0\) and \eqref{eq:wc-delta-identity} imply
\(\tau_1>\bar b_1\ge\min_{i\in S_1}b_i\). Since \(D>0\), \(\tau_2
<
\max_{i\in S_2}b_i
\le
\max_{i\in S_1}b_i\).
Thus \(\Delta_1=\tau_2-\tau_1<L_b\).
Combining this inequality with \eqref{eq:wc-delta-iteration} and \eqref{eq:wc-delta-gap} proves the result.
\end{proof}

Theorem~\ref{thm:worst-case-necessary} provides a necessary condition for a long removal phase, characterized by bounded deletions and strictly positive threshold increments. If \(K\) and \(\xi\) are bounded and \(s=\Theta(m)\), then the upper bound on \(\delta/L_b\) decays as \(\exp\bigl(-\Theta(m\log m)\bigr)\). Hence, some distinct \(b_i\) values must differ by a super-exponentially small fraction of their overall range. For large \(m\), such differences may fall below the resolution of fixed-precision arithmetic, making this quadratic-time behavior difficult to observe numerically.

\begin{remark}
Consider the uniformly distributed data used in Section~\ref{sec.5}.
Suppose that the removal step begins with \(m=10^6\) active elements.
Since \(w_i\in[10,25]\), \(\xi\le2.5\). Moreover, \(b_i\in[1,15]\), so \(L_b\le14\). Because the \(b_i\) values
are stored in IEEE~754 double precision, any two distinct values in
this interval differ by at least
\[
\delta\ge2^{-52}\approx2.22\times10^{-16}.
\]

Suppose that, in each of \(s\) consecutive iterations, the threshold
increment is positive and between one and \(K=10\) elements are
removed in each iteration. For \(s=5\), Theorem~\ref{thm:worst-case-necessary} gives
\[
\delta<
14\left(1+\frac{25}{10^6-40}\right)
\frac{25^2}{(10^6-20)(10^6-30)}
\approx8.75\times10^{-9}.
\]
For \(s=7\), it gives
\[
\delta<
14\left(1+\frac{25}{10^6-60}\right)
\frac{25^4}
{(10^6-20)(10^6-30)(10^6-40)(10^6-50)}
\approx5.47\times10^{-18}.
\]
The latter contradicts \(\delta\ge2^{-52}\). Hence, with the data representation used in these
experiments, at most six consecutive iterations can satisfy the stated
conditions.
\end{remark}

Theorem~\ref{thm:worst-case-necessary} treats the case in which the
initial active set contains at least two distinct values of \(b_i\). In the special case where all \(b_i\)'s in the initial active set are equal, the removal step becomes considerably simpler. The following theorem shows that the removal step terminates without deleting any element, giving its complexity \(\Theta(m)\),
and the subsequent remark shows that both DTA and NDTA have complexity \(\Theta(n)\) when all \(b_i\) in problem \eqref{P} are equal.

\begin{theorem}\label{thm:equal-b-removal}
Let \(S_1\) be the active set at the beginning of the removal step, with
\(m=|S_1|\). Suppose that \(D>0\), \(w_i>0\), and \(b_i=\beta\), \(\forall\,i\in S_1\) for some \(\beta\in\mathbb R\). Then the removal step deletes no
element and terminates after one pass through \(S_1\). Consequently,
the time complexity of the removal step is \(\Theta(m)\).
\end{theorem}

\begin{proof}
Let \(W_1=\sum_{i\in S_1}w_i\). The threshold at the beginning of the removal step is
\[
\tau_1
=
\frac{\sum_{i\in S_1}w_i b_i-D}{W_1}
=
\frac{\beta W_1-D}{W_1}
=
\beta-\frac{D}{W_1}.
\]
Since \(D>0\) and \(W_1>0\), we have \(\tau_1<\beta=b_i\) for all \(i\in S_1\). Thus, no element satisfies the removal condition \(b_i\le\tau_1\). Consequently, no deletion is performed, \(\tau\) remains equal to \(\tau_1\), and the active set remains unchanged. The loop therefore terminates after a single pass through the \(m\) elements of \(S_1\), resulting in a time complexity of \(\Theta(m)\).

\end{proof}

\begin{remark}
Both DTA and NDTA have a time complexity of \(\Theta(n)\) when \(b_i=\beta\) for all \(i\in\mathcal N\) in the problem \eqref{P}. By Theorem~\ref{thm:equal-b-removal}, the removal loop terminates after a single pass through the active set and therefore requires \(\Theta(n)\) time. During the forward scan, each index is processed exactly once, with \(\Theta(1)\) work required for an addition or a reset. In DTA, each deferred index is reconsidered at most once, while the solution \(x^*\) is constructed in a single final pass. Hence, the overall time complexity of both algorithms is \(\Theta(n)\).
\end{remark}

\subsection{Practical Impact on Random Instances}

We implemented two versions of the algorithm under identical hardware and implementation settings: the original version with reset (DTA) and the no-reset variant (NDTA). All experiments were repeated 100 times on exponentially distributed random data. As reported in Table~\ref{table 1}, for all tested problem sizes (\(n=5\times 10^4,\dots,10^7\)), NDTA consistently runs faster than DTA. The runtime ratio \(t_2/t_1\) ranges from \(0.75\) to \(0.94\), showing that the reset mechanism does not improve computational efficiency on this class of instances and instead introduces additional overhead. Even when \(n=10^7\), NDTA remains about \(6.4\%\) faster than DTA, and this behavior is stable over 100 independent runs. The reset statistics further support this observation. The average number of resets is only \(1.2\sim 1.3\), with no clear dependence on the problem dimension. This indicates that for exponentially distributed data, the reset condition is triggered only rarely, so the active-set evolution is already well behaved without this mechanism. In such a situation, the extra work associated with reset and Step 3 appears to be the main source of the inferior runtime of DTA.

To better understand this behavior, Table~\ref{table 2} compares the evolution of the active-set size \(|\mathcal S|\) in the two versions for the case \(n=10^6\), where reset is triggered three times. The difference between the two active sets is very small throughout the computation: after Step~2, their sizes differ by only four elements, and after Step~3 of DTA, the difference is reduced to one element. Thereafter, the two versions produce exactly the same active set in every iteration of the removal step. These observations indicate that, on this class of instances, the reset mechanism has only a limited effect on the active set evolution, while its maintenance cost outweighs its practical benefit.

\begin{table}[htbp]
	\centering
	\caption{Average, maximum, and minimum values for the number of resets (\(n_r\)) and running times (ms) of DTA($t_1$) and NDTA ($t_2$) under exponential distribution.}
	\label{table 1}
	\small
	\setlength{\tabcolsep}{12pt} 
	\begin{tabular}{ccccccc}
		\toprule
		$n$ & \multicolumn{1}{c}{$5\times 10^4$} & \multicolumn{1}{c}{$10^5$} & \multicolumn{1}{c}{$5\times 10^5$} & \multicolumn{1}{c}{$10^6$} & \multicolumn{1}{c}{$5\times 10^6$} & \multicolumn{1}{c}{$10^7$} \\
		\midrule
		Avg(\(n_r\))     & 1.2 & 1.2 & 1.3 & 1.2 & 1.2 & 1.2 \\
		Max(\(n_r\))      & 4.0 & 5.0 & 4.0 & 4.0 & 4.0 & 5.0 \\
		Min(\(n_r\))      & 0.0 & 0.0 & 0.0 & 0.0 & 0.0 & 0.0 \\
		$t_1$ (ms)   & 0.670 & 1.290 & 5.940 & 12.730 & 58.580 & 121.660 \\
		$t_2$ (ms)   & 0.500 & 1.140 & 5.450 & 12.080 & 54.760 & 114.650 \\
		$t_2/t_1$    & 0.75 & 0.88 & 0.92 & 0.95 & 0.93 & 0.94 \\
		\bottomrule
	\end{tabular}
\end{table}

\begin{table}[htbp]
	\centering
	\caption{Comparison of active set sizes after each stage between DTA and NDTA under exponential distribution.}
	\label{table 2}
	\small
	\setlength{\tabcolsep}{30pt}
	\begin{tabularx}{\textwidth}{cXc}
		\toprule
		Stage & DTA & NDTA \\
		\midrule
		After Step 2       & 99837 & 99841 \\
		After Step 3       & 99840 & 99841 \\
		After Step 4, iteration 1 & 94417 & 94417 \\
		After Step 4, iteration 2 & 94263 & 94263 \\
		After Step 4, iteration 3 & 94263 & 94263 \\
		\bottomrule
	\end{tabularx}
\end{table}

In summary, the role of the reset mechanism is inherently instance-dependent. It may become inactive or offer only marginal practical improvement on certain data, and in our auxiliary numerical tests the no-reset variant is even observed to be slightly faster in such cases, presumably because it avoids the small overhead associated with maintaining reset operations. Yet the reset mechanism remains crucial for preventing the quadratic worst-case behavior exhibited by the no-reset variant on certain instances. Its primary contribution is therefore to improve the robustness of the algorithm.

\section{Numerical Experiments}\label{sec.5}

Section~\ref{sect.NRP} reports numerical results for the zero-lower-bound weighted minimum variance allocation problem (ZWMVA), which is equivalent to problem~\eqref{P}. Section~\ref{sect.NC} compares DTA (NDTA) with WMVA~\cite{17}, Secant~\cite{9}, Variable Fixing~\cite{12}, Newton~\cite{10}, Median Search~\cite{11}, Heap~\cite{33}, and Sort~\cite{32}. Section~\ref{sect.SP} further validates the proposed algorithms for the continuous relaxation of the sensor placement problem (CSP)~\cite{Frangioni11}.

All experiments were carried out on a desktop with an Intel(R) Core(TM) i7-7700 CPU at 3.60GHz and 32 GB RAM, running Windows 10. The code was implemented in C and compiled with MSVC 19.44 using the optimization flags \texttt{/O2 /GL /arch:AVX2 /fp:fast}. Execution times were measured with \texttt{QueryPerformanceCounter} and are reported in milliseconds.

Table~\ref{tab:complexity} summarizes the worst-case, expected, and practical complexity of the algorithms. Following~\cite{18}, the practical complexity refers to the observed growth rate of running time with respect to the dimension of the problem in typical instances. As indicated in the table, although DTA has worst-case complexity \(\Theta(n^2)\), its empirical running time scales essentially linearly.
\begin{table}[htbp]
    \centering
    \caption{Worst-case, expected, and practical complexity of the algorithms as functions of the data size \(n\). Here, Sort denotes quicksort with random pivot selection (see~\cite{18}).}
    \label{tab:complexity}
    \small
    \begin{tabularx}{\textwidth}{cccc}
        \toprule
        Algorithm & Worst case complexity & Expected complexity & Observed in practice \\
        \midrule
        NDTA & $\Theta(n^2)$ & -- & $\Theta(n)$ \\
        DTA & $\Theta(n^2)$ & -- & $\Theta(n)$ \\
        WMVA & $\Theta(n^2)$ & -- & $\Theta(n)$ \\
        Secant & -- & -- & $\Theta(n)$ \\
        Variable Fixing & $\Theta(n^2)$ & -- & $\Theta(n)$ \\	
        Newton & $\Theta(n^2)$ & -- & $\Theta(n)$ \\        	
        Median Search & $\Theta(n)$ & $\Theta(n)$ & $\Theta(n)$ \\
        Heap & $\Theta(n + n \log n)$ & -- & $\Theta(n + n \log n)$ \\
        Sort & $\Theta(n^2)$ & $\Theta(n \log n)$ & $\Theta(n \log n)$ \\
        \bottomrule

    \end{tabularx}
\end{table}

\subsection{Numerical Results for the Zero-Lower-Bound Weighted Minimum Variance Allocation Problem}\label{sect.NRP}

As in~\cite{17}, we consider two representative data distributions: uniform and exponential. 
The instances are generated as follows:
\begin{itemize}
    \item[(a)] \textbf{Uniform:} \(w_i \sim \mathbb{U}[10,25]\), \(b_i \sim \mathbb{U}[1,15]\), and \(c\in\left(0,\sum_i w_i b_i\right)\).
    \item[(b)] \textbf{Exponential:}
    \begin{itemize}[label=\textbullet]
        \item \(w_i\) is generated by \(-\log(1-u)\), where \(u\sim\mathbb{U}(0,1)\).
        \item \(b_i\) is generated by \(-\log(1-u)\), where \(u\sim\mathbb{U}(0,1)\).
        \item \(c = \frac{r}{4.6}\sum_i w_i b_i\), where \(r = -\log(1-u)\) with \(u\sim\mathbb{U}(0,1)\) such that \(r<4.6\).
    \end{itemize}
\end{itemize}
\smallskip
Here, the problem dimension \(n\) ranges from \(5\times 10^4\) to \(10^7\).

\subsubsection{Numerical Comparisons}\label{sect.NC}

We next compare the numerical performance of the algorithms, focusing on illustrating the iteration process and the computational efficiency. As the first step, Table~\ref{table 4} reports the evolution of the active set for DTA, NDTA, WMVA, Variable Fixing, and Newton on a representative instance under the exponential distribution. The remaining methods are omitted because they do not maintain active sets and are therefore not iterative in this sense. Although the stopping criteria differ, each iteration essentially amounts to processing the current active set. The table shows that NDTA and DTA terminate in three iterations, whereas WMVA, Variable Fixing, and Newton require six. In addition, Variable Fixing, WMVA, and Newton produce identical active set sizes at every iteration, consistent with \cite[Proposition~1]{10} and \cite[Theorem~3.1]{22}. However, these methods differ substantially in per-iteration cost and overall running time, as will be seen more clearly in the timing results below.

\begin{table}[htbp]
	\centering
	\caption{Active set size \(|\mathcal{S}|\) after each iteration for a representative instance with \(n=10^6\) under the exponential distribution.}
	\label{table 4}
	\small
	\setlength{\tabcolsep}{12pt}
	\begin{tabularx}{\textwidth}{cccccc}
		\toprule
		Iteration & NDTA & DTA & WMVA & Variable Fixing & Newton \\
		\midrule
		0 & 695181 & 695178 & 1000000 & 1000000 & 1000000  \\
		1 & 563118 & 563117 & 674969 & 674969 & 674969     \\
		2 & 479502 & 479502 & 521421 & 521421 & 521421     \\
		3 & 479502 & 479502 & 505601 & 505601 & 505601     \\
		4 &        &        & 480244 &  480244 & 480244     \\
		5 &        &        & 479502 & 479502  & 479502      \\
		6  &       &        & 479502 & 479502 & 479502       \\
		\bottomrule
	\end{tabularx}
\end{table}

\begin{sidewaystable}[htbp]
    \centering
    \renewcommand\arraystretch{1.1}
\caption{Average, maximum, and minimum iterations and running times of NDTA, DTA, WMVA, Secant, Variable Fixing, Newton, Median Search, Heap, and Sort under the uniform distribution.}
    \label{table 5}
    \setlength{\tabcolsep}{3pt} 
    \begin{tabular}{c ccc ccc | ccc ccc | ccc ccc}
        \toprule        
        Dimension& \multicolumn{3}{c}{Iteration} & \multicolumn{3}{c}{Time (ms)} 
        & \multicolumn{3}{c}{Iteration} & \multicolumn{3}{c}{Time (ms)} 
        & \multicolumn{3}{c}{Iteration} & \multicolumn{3}{c}{Time (ms)} \\
        \cmidrule(lr){2-4} \cmidrule(lr){5-7} \cmidrule(lr){8-10} \cmidrule(lr){11-13} \cmidrule(lr){14-16} \cmidrule(lr){17-19}       
        $n$& Avg. & Max. & Min. & Avg. & Max. & Min. & Avg. & Max. & Min. & Avg. & Max. & Min. & Avg. & Max. & Min. & Avg. & Max. & Min. \\
        \midrule
          & \multicolumn{5}{c}{NDTA} && \multicolumn{5}{c}{DTA} && \multicolumn{5}{c}{WMVA}& \\     
        $5\times10^4$   & 4.6 & 5 & 4 & 0.541 & 1.234 & 0.518 
        & 4.6 & 5 & 4 & 0.636 & 1.313 & 0.534 
        & 4.0 & 5 & 4 & 0.818 & 1.436 & 0.651 \\
        $10^5$          & 4.7 & 6 & 4 & 1.203 & 2.118 & 1.041 
        & 4.7 & 6 & 4 & 1.240 & 2.183 & 1.084 
        & 4.1 & 5 & 4 & 1.493 & 2.774 & 1.309 \\
        $5\times 10^5$  & 4.8 & 6 & 4 & 7.125 & 11.063& 6.037 
        & 4.8 & 6 & 4 & 7.490 & 11.611& 6.286 
        & 4.1 & 5 & 4 & 9.411 & 13.755& 8.655 \\
        $10^6$          & 5.0 & 6 & 4 & 14.414& 17.760& 13.241
        & 5.0 & 6 & 4 & 15.375& 18.412& 13.857
        & 4.2 & 6 & 4 & 18.977& 27.934& 17.899 \\
        $5\times 10^6$  & 5.0 & 6 & 4 & 73.589& 87.414& 70.581
        & 5.0 & 6 & 4 & 77.442& 89.717& 71.877
        & 4.1 & 5 & 4 & 99.752& 122.270& 91.407 \\
        $10^7$   & 4.7 & 6 & 4 & 150.827& 179.763& 143.539
        & 4.7 & 6 & 4 & 160.411& 187.252& 148.531
        & 4.1 & 5 & 4 & 210.517& 273.885& 189.526 \\
        \midrule
        & \multicolumn{5}{c}{Secant} && \multicolumn{5}{c}{Variable Fixing} &&\multicolumn{5}{c}{Newton}& \\
        $5\times10^4$   & 7.0 & 8 & 7 & 0.885 & 2.686 & 0.633 
        & 4.0 & 5 & 4 & 1.105 & 2.820 & 0.870 
        & 4.0 & 5 & 4 & 1.234 & 3.930 & 0.920 \\
        $10^5$          & 7.0 & 7 & 7 & 1.798 & 5.724& 1.308 
        & 4.1 & 5 & 4 & 2.449 & 5.250& 1.842 
        & 4.1 & 5 & 4 & 2.397 & 4.677 & 1.955 \\
        $5\times 10^5$  & 7.1 & 8 & 7 & 11.958& 15.999& 9.969 
        & 4.1 & 5 & 4 & 14.851& 19.443& 12.579
        & 4.1 & 5 & 4 & 15.891& 19.161& 12.812 \\
        $10^6$          & 8.0 & 8 & 7 & 24.187& 31.129& 21.080
        & 4.2 & 6 & 4 & 29.938& 34.983& 26.771
        & 4.2 & 6 & 4 & 31.207& 35.471& 27.717 \\
        $5\times 10^6$  & 7.0 & 7 & 7 & 113.249& 144.77 & 105.599
        & 4.1 & 5 & 4 & 152.714& 179.747& 140.714 
        & 4.1 & 5 & 4 & 160.400 &193.869 & 147.772\\
        $10^7$  & 7.1 & 8 & 7 & 236.458& 299.149&209.504
        & 4.1 & 5 & 4 & 320.183& 388.774 & 288.145 
        & 4.1 & 5 & 4 & 342.048& 498.097& 306.078 \\
        \midrule
        & \multicolumn{5}{c}{Median Search} && \multicolumn{5}{c}{Heap} && \multicolumn{5}{c}{Sort}& \\
        $5\times10^4$   & 16.5 & 17 & 16 & 1.402 & 1.711 & 1.264 
        & --- & --- & --- & 3.261 & 8.226 & 2.583 
        & --- & --- & --- & 11.596 & 28.494 & 7.272 \\
        $10^5$       & 17.5 & 18 & 17 & 2.780 & 3.483 & 1.791 
        & --- & --- & --- & 7.006 & 15.066& 5.816 
        & --- & --- & --- & 22.077 & 32.397 & 16.217 \\
        $5\times 10^5$  & 19.9 & 20 & 19 & 15.278& 23.850& 12.240
        & --- & --- & --- & 50.697& 58.255& 45.872
        & --- & --- & --- & 104.744&141.558& 85.336 \\
        $10^6$          & 20.9 & 21 & 20 & 36.294& 42.544& 31.153
        & --- & --- & --- & 133.119&161.741&125.259
        & --- & --- & --- & 210.773&238.237&181.791 \\
        $5\times 10^6$  & 23.2 & 24 & 23 & 173.012& 225.999& 154.011
        
        & --- & --- & --- & 1017.564& 1248.571& 993.454 
       & --- & --- & --- & 1178.018 & 1332.870& 1083.569 \\
        $10^7$  & 24.2 & 25 & 24 & 364.155& 381.227& 342.982
        
        & --- & --- & --- & 2192.294& 2491.026& 2061.062
        & --- & --- & --- & 2290.576& 2528.057& 2104.854 \\
        \bottomrule
    \end{tabular}
\end{sidewaystable}

\begin{sidewaystable}[htbp]
    \centering
    \renewcommand\arraystretch{1.1}
    \caption{Average, maximum, and minimum iterations and running times of NDTA, DTA, WMVA, Secant, Variable Fixing, Newton, Median Search, Heap, and Sort under the exponential distribution.}
    \label{table 6}
    \setlength{\tabcolsep}{3pt} 
    \begin{tabular}{c ccc ccc | ccc ccc | ccc ccc}
        \toprule        
        Dimension& \multicolumn{3}{c}{Iteration} & \multicolumn{3}{c}{Time (ms)} 
        & \multicolumn{3}{c}{Iteration} & \multicolumn{3}{c}{Time (ms)} 
        & \multicolumn{3}{c}{Iteration} & \multicolumn{3}{c}{Time (ms)} \\
        \cmidrule(lr){2-4} \cmidrule(lr){5-7} \cmidrule(lr){8-10} \cmidrule(lr){11-13} \cmidrule(lr){14-16} \cmidrule(lr){17-19}       
        $n$& Avg. & Max. & Min. & Avg. & Max. & Min. & Avg. & Max. & Min. & Avg. & Max. & Min. & Avg. & Max. & Min. & Avg. & Max. & Min. \\
        \midrule
          & \multicolumn{5}{c}{NDTA} && \multicolumn{5}{c}{DTA} && \multicolumn{5}{c}{WMVA}& \\        
        $5\times10^4$   & 4.7 & 7 & 3 & 0.234 & 0.778 & 0.227 
        & 4.7 & 7 & 3 & 0.313 & 0.934 & 0.242 
        & 3.7 & 7 & 2 & 0.473 & 0.916 & 0.396 \\
        $10^5$          & 4.3 & 5 & 4 & 0.646 & 2.289 & 0.568 
        & 4.3 & 5 & 4 & 0.687 & 2.435 & 0.604 
        & 3.3 & 4 & 3 & 0.931 & 3.433 & 0.840 \\
        $5\times 10^5$  & 3.8 & 6 & 2 & 5.102 & 8.971 & 4.630 
        & 3.8 & 6 & 2 & 5.747 & 9.756 & 4.926 
        & 3.8 & 7 & 2 & 7.714 & 18.366 & 6.702 \\
        $10^6$          & 4.0 & 6 & 2 & 8.118 & 11.143 & 7.053 
        & 4.0 & 6 & 2 & 8.636 & 12.046 & 7.508 
        & 3.7 & 6 & 2 & 10.918 & 15.314 & 9.426 \\
        
        $5\times 10^6$  & 4.5 & 6 & 2& 64.772& 94.675& 33.468
        & 4.5 & 6 & 2 & 68.157& 97.146& 38.265
        & 4.1 & 6 & 2 & 94.132 & 138.262& 41.500 \\
        $10^7$  & 4.6 & 6 & 3 & 122.337& 182.734& 91.114
        & 4.6 & 6 & 3 & 127.435& 187.305& 94.303
        & 3.6 & 5 & 2 & 170.706& 312.851& 86.644\\
        \midrule
        & \multicolumn{5}{c}{Secant} && \multicolumn{5}{c}{Variable Fixing} &&\multicolumn{5}{c}{Newton}& \\
        $5\times10^4$   & 7.6 & 10 & 5 & 0.761 & 2.914 & 0.560 & 3.7 & 7 & 2 & 0.520 & 1.227 & 0.396 & 3.7 & 7 & 2 & 0.558 & 1.143 & 0.430 \\
        $10^5$          & 7.1 & 8 & 6 & 1.657 & 4.267 & 1.341 & 3.3 & 4 & 3 & 1.489 & 4.360 & 1.181 & 3.3 & 4 & 3 & 1.502 & 3.807 & 1.241 \\
        $5\times 10^5$  & 7.6 & 10 & 5 & 11.638 & 17.836 & 9.883 & 3.8 & 7 & 2 & 12.579 & 25.022 & 10.630 & 3.8 & 7 & 2 & 13.348 & 25.000 & 11.323 \\
        $10^6$          & 7.6 & 10 & 5 & 16.419 & 23.800 & 15.001 & 3.7 & 6 & 2 & 17.964 & 25.540 & 16.149 & 3.7 & 6 & 2 & 19.456 & 27.994 & 17.475 \\
        $5\times 10^6$  & 7.8 & 9 & 5 & 131.179& 169.300& 81.077
        & 4.1 & 6 & 2 & 150.921& 210.582& 79.663
        & 4.1 & 6 & 2 & 156.594  & 213.588& 75.037 \\
        $10^7$  & 7.6 & 9 & 6 & 252.524& 330.387& 179.250
         & 3.6 & 5 & 2 & 263.168 & 394.830& 144.784
        & 3.6 & 5 & 2 & 280.930& 429.609 & 207.467\\
        \midrule
        & \multicolumn{5}{c}{Median Search} && \multicolumn{5}{c}{Heap} && \multicolumn{5}{c}{Sort}& \\
        $5\times10^4$   & 16.4 & 17 & 16 & 3.298 & 7.223 & 2.634         
        & --- & --- & --- & 7.485 & 10.795 & 6.239 
        & --- & --- & --- & 12.434 & 25.886 & 7.482 \\
        $10^5$      & 17.5 & 18 & 17 & 5.790 & 7.716 & 5.146        
        & --- & --- & --- & 12.010 & 15.251 & 10.731 
        & --- & --- & --- & 22.383 & 42.269 & 15.412 \\
        $5\times 10^5$  & 19.9 & 20 & 19 & 34.020 & 39.551 & 31.718 
        
        & --- & --- & --- & 88.657 & 98.782 & 74.511 
        & --- & --- & --- & 104.424& 151.578& 83.534 \\
        $10^6$          & 20.9 & 21 & 20 & 65.906 & 75.587 & 62.644       
        & --- & --- & --- & 185.615 & 249.708 & 162.390 
        & --- & --- & --- & 208.854& 303.230& 178.191 \\
        
        $5\times 10^6$  & 23.3  & 24 & 23 & 316.152& 343.961& 290.865
        
        & --- & --- & --- & 1023.910& 1254.082& 937.081
        & --- & --- & --- & 1108.785  & 1363.138& 1036.507 \\
        $10^7$  &24.2  & 25 & 24 & 634.588& 681.517& 617.730 
        
        & --- & --- & --- &   1904.134  &  2164.312& 1879.974
        & --- & --- & --- &  2340.354&  2478.503 &  2218.861\\
        \bottomrule
    \end{tabular}
\end{sidewaystable}

The numerical results are reported in Tables~\ref{table 5} and~\ref{table 6}. For each dimension of the problem, the tables list the iteration counts (average, maximum, and minimum) and the running times over \(100\) random instances. To obtain stable timing measurements, each instance was solved repeatedly $\max\{10,10^7/n\}$ times, and the reported time is the average per instance. The notion of an iteration depends on the algorithm. For DTA and NDTA, one iteration corresponds to one complete pass of Step~4. For WMVA, Secant, Variable Fixing, and Newton, each evaluation of \(g(\tau)\) in~\eqref{eq:xgtau} is counted as one iteration. For Median Search, one iteration corresponds to one median-search operation. Sort and Heap do not admit a meaningful iteration count and are therefore compared only in terms of running time. Although Median Search has favorable theoretical complexity $\Theta(n)$, its observed running time is relatively large in our setting, likely due to the overhead of repeated scans and median selection.

Tables~\ref{table 5} and~\ref{table 6} show that DTA and NDTA are highly competitive across all dimensions tested and in both data distributions. For example, under the exponential distribution with \(n=10^7\), the average running times are \(127.435\)ms for DTA and \(122.337\)ms for NDTA, compared with \(170.706\)ms for WMVA~\cite{17}, \(252.524\)ms for Secant, \(263.168\)ms for Variable Fixing, \(634.588\)ms for Median Search, and \(280.930\)ms for Newton. The advantage over Heap and Sort is even more pronounced, with running times of \(1.904\)s and \(2.340\)s, respectively. The same overall trend is observed under the uniform distribution. The tables also support the practical linear-time behavior of DTA and NDTA. Under the uniform distribution, as \(n\) increases from \(5\times 10^4\) to \(10^7\), the average running time grows from \(0.636\)ms to \(160.411\)ms for DTA and from \(0.541\)ms to \(150.827\)ms for NDTA, indicating near-linear scaling in both cases. Moreover, NDTA is consistently slightly faster than DTA on these instances, suggesting that the reset operation does not provide an additional practical benefit here and may instead introduce a small overhead. This observation indicates that in such settings, DTA can be simplified to NDTA to achieve slightly better practical efficiency.

In terms of iteration counts, WMVA, Newton, and Variable Fixing have the same number of iterations, typically requiring about \(3\) to \(4\) iterations on average. Both DTA and NDTA require slightly more iterations, usually about \(4\) to \(5\), with a maximum of \(7\), but this number remains essentially stable as \(n\) increases. By contrast, Secant typically requires about \(7\) to \(8\) iterations, and Median Search about \(17\) to \(24\) iterations. It is worth highlighting that NDTA and DTA achieve the shortest running time in all test scenarios despite requiring slightly more iterations than WMVA, Newton, and Variable Fixing. This reflects its substantially lower per-iteration cost: each NDTA or DTA iteration involves only simple scalar comparisons and updates of the dynamic threshold, whereas the competing methods require more expensive operations such as derivative evaluation, set updates, or search procedures.

Overall, DTA and NDTA exhibit stable and efficient performance under both uniform and exponential distributions, with running times that grow approximately linearly in the size of the problem. Across all comparisons, they provide the best overall computational efficiency among the tested methods.

\subsection{Numerical Results for Sensor Placement Problem}\label{sect.SP}

\begin{sidewaystable}[htbp]
    \centering
    \renewcommand\arraystretch{1.1}
    \caption{Numerical results for the sensor placement problem: numbers of failed runs, together with the average, maximum and minimum iterations and running times of NDTA, DTA, WMVA, FPA2, Newton, Variable Fixing, Median Search, and Secant.}
    \label{tab:SP}
    \setlength{\tabcolsep}{3 pt}
    \footnotesize
    \begin{tabular}{c ccc ccc c | ccc ccc c}
        \toprule
        \multicolumn{1}{c}{Dimension} & \multicolumn{3}{c}{Iteration} & \multicolumn{3}{c}{Time (ms)} & \multicolumn{1}{c}{Failures} & \multicolumn{3}{c}{Iteration} & \multicolumn{3}{c}{Time (ms)} & \multicolumn{1}{c}{Failures} \\
        \cmidrule(lr){2-4} \cmidrule(lr){5-7} \cmidrule(lr){9-11} \cmidrule(lr){12-14}
        $n$& Avg. & Max. & Min. & Avg. & Max. & Min. & --& Avg. & Max. & Min. & Avg. & Max. & Min. &-- \\
                \midrule
         & \multicolumn{5}{c}{NDTA} &&& \multicolumn{5}{c}{DTA} &\\
        $10^4$-A   & 1.0 & 1 & 1 & 0.023 & 0.037 & 0.022 & 0 
        & 1.0 & 1 & 1 & 0.033 & 0.045 & 0.024 & 0 \\
        $10^4$-B   & 1.5 & 2 & 1 & 0.034 & 0.062 & 0.031 & 0 
        & 1.5 & 2 & 1 & 0.049 & 0.071 & 0.035 & 0 \\
        $10^5$-A  & 1.0 & 1 & 1 & 0.173 & 0.232 & 0.150 & 0 
        & 1.0 & 1 & 1 & 0.204 & 0.269 & 0.167 & 0 \\
        $10^5$-B  & 1.9 & 2 & 1 & 0.209 & 0.273 & 0.190 & 0 
        & 1.9 & 2 & 1 & 0.266 & 0.348 & 0.221 & 0 \\
        $10^6$-A & 1.2 & 2 & 1 & 1.977 & 2.416 & 1.443 & 0 
        & 1.2 & 2 & 1 & 2.222 & 2.795 & 1.826 & 0 \\
        $10^6$-B & 1.7 & 2 & 1 & 2.178 & 2.406 & 1.870 & 0 
        & 1.7 & 2 & 1 & 2.420 & 2.673 & 2.078 & 0 \\
        $5\times10^6$-A & 1.0 & 1 & 1 & 8.710 & 9.278& 7.747 & 0 
        & 1.0 & 1 & 1 & 9.678 & 11.420& 8.667 & 0 \\
        $5\times10^6$-B & 1.2 & 2 & 1 & 9.675 & 10.240& 7.561 & 0 
        & 1.2 & 2 & 1 & 10.393& 12.489& 8.845 & 0 \\
        $10^7$-A& 1.1 & 2 & 1 & 16.935& 17.865& 16.361& 0 
        & 1.1 & 2 & 1 & 18.594& 19.406& 17.845& 0 \\
        $10^7$-B& 1.1 & 2 & 1 & 17.943& 19.516& 16.495& 0 
        & 1.1 & 2 & 1 & 19.826& 22.795& 18.217& 0 \\
        \midrule
        & \multicolumn{5}{c}{WMVA} &&& \multicolumn{5}{c}{FPA2} & \\
        $10^4$-A   & 6.1 & 7 & 5 & 0.130 & 0.205 & 0.098 & 0 
        & 3.9 & 5 & 3 & 0.089 & 0.118 & 0.077 & 0 \\
        $10^4$-B   & 13.4& 15& 11& 0.098 & 0.130 & 0.081 & 0 
        & 7.4 & 8 & 6 & 0.167 & 0.346 & 0.092 & 0 \\
        $10^5$-A  & 5.7 & 7 & 5 & 1.142 & 1.531 & 0.880 & 0 
        & 4.3 & 5 & 3 & 0.813 & 1.014 & 0.739 & 0 \\
        $10^5$-B  & 14.4& 17& 9 & 0.857 & 1.050 & 0.699 & 0 
        & 9.0 & 11& 5 & 1.274 & 1.556 & 1.087 & 0 \\
        $10^6$-A & 6.2 & 8 & 5 & 16.609 & 20.399& 12.946& 0 
        & 4.3 & 5 & 4 & 7.961 & 9.308 & 6.745 & 0 \\
        $10^6$-B & 17.5& 20& 10& 11.088 & 12.981& 9.691 & 0 
        & 10.3& 13& 8 & 13.844& 15.423& 11.700& 1 \\
        $5\times10^6$-A & 4.8 & 7 & 3 & 76.425 & 85.246& 68.146& 2 
        & 0.0 & 0 & 0 & --- & --- & --- & 10 \\
        $5\times10^6$-B & 20.7& 22& 17& 51.439 & 58.402& 47.012& 0 
        & 11.0& 14& 8 & 62.546& 67.567& 56.539& 2 \\
        $10^7$-A& 5.4 & 8 & 4 & 148.900 & 193.709&128.837& 1 
        & 0.0 & 0 & 0 & --- & --- & --- & 10 \\
        $10^7$-B& 18.6& 21& 15 & 102.438 &115.805& 95.374& 1 
        & 13.1& 16& 9 & 137.523 & 144.650 & 108.603 & 1 \\
        \midrule
         & \multicolumn{5}{c}{Newton} &&& \multicolumn{5}{c}{Variable Fixing} & \\
        $10^4$-A   & 5.6 & 6 & 4 & 0.122 & 0.176 & 0.103 & 0 
        & 5.6 & 6 & 4 & 0.126 & 0.146 & 0.091 & 0 \\
        $10^4$-B   & 4.3 & 6 & 3 & 0.215 & 0.488 & 0.132 & 0 
        & 4.3 & 6 & 3 & 0.195 & 0.448 & 0.125 & 0 \\
        $10^5$-A  & 6.0 & 8 & 3 & 1.128 & 1.430 & 0.971 & 0 
        & 6.0 & 8 & 3 & 0.837 & 1.050 & 0.674 & 0 \\
        $10^5$-B  & 5.3 & 7 & 4 & 1.608 & 1.995 & 1.420 & 0 
        & 5.3 & 7 & 4 & 1.190 & 1.408 & 1.013 & 0 \\
        $10^6$-A & 7.6 & 10& 3 & 9.947 & 12.614& 8.829 & 0 
        & 7.6 & 10& 3 & 11.135& 12.406& 10.250& 0 \\
        $10^6$-B & 7.3 & 9 & 6 & 17.595& 19.758& 16.027& 0 
        & 6.7 & 9 & 6 & 17.626& 21.727& 14.947& 2 \\
        $5\times10^6$-A & 21.1& 29& 11& 45.152& 51.017& 42.126& 0 
        & 0.0 & 0 & 0 & --- & --- & --- & 10 \\
        $5\times10^6$-B & 9.0 & 11& 6 & 80.895& 88.448& 74.711& 0 
        & 8.0 & 10& 6 & 81.996 & 86.845& 77.145& 8 \\
        $10^7$-A& 23.6& 39& 8 & 87.677& 90.426& 85.412& 0 
        & 0.0 & 0 & 0 & --- & --- & --- & 10 \\
        $10^7$-B& 10.0& 11& 8 & 160.990&170.754&148.771& 0 
        & 0.0 & 0 & 0 & --- & --- & --- & 10 \\
         \midrule
         & \multicolumn{5}{c}{Median Search} &&& \multicolumn{5}{c}{Secant} &\\
        $10^4$-A   & 14.3& 15& 14& 0.429 & 0.511 & 0.334 & 0 & 13.4& 18& 9 & 0.305 & 0.552 & 0.179 & 0 \\
        $10^4$-B   & 14.6& 15& 14& 0.550 & 1.291 & 0.365 & 0 & 16.5& 20& 15& 0.456 & 0.673 & 0.305 & 0 \\
        $10^5$-A  & 17.5& 18& 17& 3.136 & 3.965 & 2.638 & 0 & 16.4& 22& 13& 2.818 & 4.691 & 1.942 & 0 \\
        $10^5$-B  & 17.5& 18& 17& 3.473 & 4.822 & 2.981 & 0 & 22.0& 29& 16& 4.095 & 5.948 & 2.525 & 0 \\
        $10^6$-A & 21.0& 21& 21& 35.646& 37.225& 34.254& 0 & 26.0& 31& 24& 69.648& 88.189& 57.293& 0 \\
        $10^6$-B & 21.0& 21& 21& 40.409& 42.873& 37.891& 0 & 26.2& 31& 20& 73.012& 85.189& 50.023& 0 \\
        $5\times10^6$-A & 23.1& 24& 23& 162.209&174.564&155.661& 0 & 31.6& 38& 27& 397.853&506.582&350.901& 0 \\
        $5\times10^6$-B & 23.3& 24& 23& 173.998&208.180&160.876& 0 & 31.2& 42& 26 & 445.526&549.842&307.450& 1 \\
        $10^7$-A& 24.5& 25& 24& 314.219&326.259&299.808& 0 & 36.1& 46& 30& 901.595&1126.334&745.351& 0 \\
        $10^7$-B& 24.2& 25& 24& 352.367&369.425&328.137& 0 & 33.2& 46& 26& 888.466&1196.240&691.087& 0 \\
        \bottomrule
    \end{tabular}
\end{sidewaystable}

In this section, following~\cite{19}, we test DTA and NDTA on the sensor placement problem arising in traffic monitoring. Instances are generated by the public generator at \texttt{http://groups.di.unipi.it/optimize/Data/RDR.html}, which produces two classes of test problems, denoted Type-A and Type-B. We consider \(n=10^4\), \(10^5\), \(10^6\), \(5\times 10^6\), and \(10^7\). For each size, ten random instances are generated, and each instance is solved ten times to obtain stable timing estimates. Reported times are in milliseconds. For FPA2~\cite{19}, the stopping criterion is \(|\tau_k-\tau_{k-1}| \le tol\), where \(\tau_k\) is the threshold at the end of the $k$-th iteration and \(tol=10^{-12}\) as in~\cite{10}.

Table~\ref{tab:SP} reports running times, iteration counts, and numbers of failed runs for DTA, NDTA, WMVA~\cite{17}, FPA2~\cite{19}, Newton~\cite{10}, Variable Fixing~\cite{12}, Median Search~\cite{11}, and Secant~\cite{9}. A run is counted as failed if the iteration count reaches \(100\) or if
\[
\left|\frac{g(\tau)-c}{c}\right| \ge 10^{-8}.
\]
Reported averages are computed over successful runs only.

The results indicate that all algorithms are reliable on small and medium-scale instances (\(n\le 10^6\)), with only a few failures on \(10^6\)-B: one for FPA2 and two for Variable Fixing. In large-scale instances (\(n=5\times 10^6\) and \(n=10^7\)), their robustness diverges substantially. NDTA and DTA, both based on nondecreasing threshold updates, solved every test instance, as did Newton and Median Search. WMVA failed on a few instances and Secant on one, whereas FPA2 and Variable Fixing failed on all large type-A instances and many type-B instances.

From an efficiency perspective, NDTA and DTA are the most competitive, requiring only \(1\) to \(2\) iterations on average and less than \(20\)ms even in the largest instances. WMVA uses more iterations but maintains a low per-iteration cost, yielding a total running time comparable to Newton. Newton and Secant require more iterations overall, while Median Search has the largest iteration counts despite perfect robustness.

Overall, these results confirm that NDTA and DTA offer the best balance between efficiency and robustness, especially on large-scale instances. NDTA and DTA remain failure-free, whereas some competing methods suffer from either increased iteration counts or significant robustness issues.

\section{Conclusions and Future Work}\label{sec.6}

In this paper, we studied problem~\eqref{P} and two related applications, namely, the zero-lower-bound weighted minimum variance allocation problem~\eqref{ZWMVA} and the continuous relaxation of the sensor placement problem~\eqref{CSP}. We proposed the DTA, established its correctness, and analyzed its computational complexity. Our analysis also clarified the roles of the removal and reset operations. In particular, we derived a sufficient condition under which resets do not occur, which gives rise to the no-reset variant NDTA.  Constructed instances show that resetting can reduce the
running time from \(\Theta(n^2)\) to \(\Theta(n)\), although both
algorithms can attain the quadratic worst case. Under bounded weight ratios and bounded deletions per iteration, this worst-case removal pattern requires the smallest difference between distinct input values to decrease super-exponentially relative to their overall range, making it difficult to realize in fixed-precision arithmetic. Numerical experiments on instances with up to \(10^7\) variables show near-linear scaling for both DTA and NDTA. Their running times are
comparable, with NDTA occasionally being slightly faster. Together
with the constructed examples, these results show that the effect of
resetting depends on the input. It has little influence on typical
instances but can prevent a substantial slowdown on unfavorable ones. Overall, the proposed framework yields efficient and robust solvers for large-scale weighted continuous quadratic knapsack problems, with clear advantages over existing methods in running time.

Several directions merit further study, including extensions to more general weighted minimum variance allocation problems, broader classes of quadratic objectives, and models with multiple equality constraints.

\section*{Acknowledgments}
The research of Yong-Jin Liu was supported by the National Key Research and Development Program of China (2025YFA1016901), the National Natural Science Foundation of China (Grant Nos. 12271097, 12671357), the Key Program of National Science Foundation of Fujian Province of China (Grant No. 2023J02007), the Central Guidance on Local Science and Technology Development Fund of Fujian Province (Grant No. 2023L3003). The research of Chuan Yang was supported by the National Natural Science Foundation of China (Grant No. 12526521) and the Natural Science Foundation of Fujian Province (Grant No. 2026J008061).

\renewcommand\refname{References}

\end{document}